\documentclass[11pt]{article}

\usepackage[latin1]{inputenc}
\usepackage[english]{babel}
\usepackage{amssymb,amsmath}
\usepackage{amsthm}
\allowdisplaybreaks
\usepackage{amsfonts}
\usepackage{graphicx}
\usepackage[margin=2.75cm]{geometry}
\usepackage{enumerate}
\usepackage{color}
\usepackage[noadjust]{cite}
\usepackage{tikz}
\usepackage{mathtools}
\usepackage{framed}
\usepackage{subfig}
\usepackage{epstopdf}
\usepackage{cite}
\usepackage{algorithm, algorithmic}

\usepackage{hyperref}

\newtheorem{thm}{Theorem}

\newtheorem{cor}[thm]{Corollary}

\theoremstyle{definition}

\newtheorem*{conj*}{Conjecture}

\newcommand{\maxM}{\ensuremath{\nu}}
\newcommand{\expM}{\ensuremath{\mu}}

\newcommand{\Mset}{\ensuremath{\textbf{\textit{M}}}}
\newcommand{\Tset}{\ensuremath{\textbf{\textit{T}}}}

\newcommand{\Sk}[2]{\ensuremath{S(#1, #2)}}
\newcommand{\Skw}[2]{\ensuremath{S^{ARW}(#1, #2)}}

\newcommand{\nM}{\ensuremath{\mathcal{T}_0}}
\newcommand{\totM}{\ensuremath{\mathcal{T}_1}}
\newcommand{\nMw}{\ensuremath{\mathcal{T}_0^{ARW}}}
\newcommand{\totMw}{\ensuremath{\mathcal{T}_1^{ARW}}}

\newcommand{\nMo}{\ensuremath{\mathcal{T}_0^{o}}}
\newcommand{\totMo}{\ensuremath{\mathcal{T}_1^{o}}}
\newcommand{\avM}{\ensuremath{\mathcal{I}}}

\newcommand{\avMord}{\ensuremath{\mathcal{I}^{o}}}
\newcommand{\dyer}{\ensuremath{\mathcal{I}^{DF}}}
\newcommand{\wagner}{\ensuremath{\mathcal{I}^{ARW}}}

\newcommand{\PHOEG}{\emph{PHOEG}}

{% begin code
  \begin{framed}\sffamily \textbf{#1}%
}%
{\rmfamily\end{framed}}% end code 

\newcommand{\Path}[1]{\ensuremath{{\sf P}_{#1}}}
\newcommand{\PathT}[1]{\widetilde{\ensuremath{{\sf P}}}_{#1}}

\newcommand{\C}[1]{\ensuremath{{\sf C}_{#1}}}
\newcommand{\CT}[1]{\widetilde{\ensuremath{{\sf C}}}_{#1}}

\newcommand{\W}[1]{\ensuremath{{\sf W}_{#1}}}
\newcommand{\K}[1]{\ensuremath{{\sf K}_{#1}}}
\newcommand{\KT}[1]{\widetilde{\ensuremath{{\sf K}}}_{#1}}
\newcommand{\GT}[1]{\widetilde{\ensuremath{{#1}}}}

\usetikzlibrary{shapes, graphs, graphs.standard, calc}
\tikzstyle{vertex}=[circle, draw, fill=black, minimum size=8pt, inner sep=0pt]

\makeatletter  

\title{The average size of maximal matchings in graphs}

\author{
		Alain Hertz\textsuperscript{1},
		S\'ebastien Bonte\textsuperscript{2},
		Gauvain Devillez\textsuperscript{2},
		Hadrien M\'elot\textsuperscript{2}	
		\\[3mm]
		\footnotesize \textsuperscript{1} Department of Mathematics and Industrial
	Engineering\\
	\footnotesize Polytechnique Montr\'eal - Gerad, Montr\'eal, Canada\\
	\footnotesize Corresponding author. Email: alain.hertz@gerad.ca\\[3mm]
	 \footnotesize \textsuperscript{2} Computer Science Department - Algorithms Lab\\
	 \footnotesize University of Mons, Mons, Belgium
	}

\begin{document}

\maketitle
\vspace*{0.2cm}

\hrule
\vspace*{0.2cm}
\small
\noindent
\textbf{Abstract.} \\

We investigate the ratio $\avM(G)$ of the average size of a maximal matching to the size of a maximum matching in a graph $G$. If many maximal matchings have a size close to $\maxM(G)$, this graph invariant has a value close to 1. Conversely, if many maximal matchings have a small size, $\avM(G)$ approaches $\frac{1}{2}$. 

We propose a general technique to determine the asymptotic behavior of $\avM(G)$ for various classes of graphs. To illustrate the use of this technique, we first show how it makes it possible to find known asymptotic values of $\avM(G)$ which were typically obtained using generating functions, and we then determine the asymptotic value of $\avM(G)$ for other families of graphs, highlighting the spectrum of possible values of this graph invariant between $\frac{1}{2}$ and $1$. 

\vspace*{0.4cm}
\noindent
\emph{Keywords:} Maximal matching, average size, graph invariant, asymptotic value.

\vspace*{0.2cm}
\hrule

\normalsize

\section{Introduction} \label{sec_intro}

 A \emph{matching} in a graph $G=(V,E)$ is a set $M \subseteq E$ of edges without common vertices. If $M$ is not a subset of any other matching in $G$, then it is \emph{maximal} and a matching of maximum size is called a \emph{maximum matching}. Maximal matchings are also called  \emph{independent edge dominating sets}  \cite{YG80}. Obviously, all maximum matchings are maximal but the converse does not always hold. Matching theory is a core subject in graph theory and has found applications in several domains such as networks, social science or chemistry (see  the book of Lov\'asz and Plummer~\cite{lovasz09} for a complete overview of matching theory). The main objective of this paper is to answer the following natural question. 

\begin{quote}
\emph{Given a graph, are there many maximal matchings that are significantly different in size than a maximum matching?}
\end{quote}

Let $\maxM(G)$ be the size of a maximum matching in $G$, let $\Mset(G)$ be the set of maximal matchings in $G$, let $\nM(G)=|\Mset(G)|$ be the number of maximal matchings in $G$ and let $\totM(G)$ be the sum of the sizes of all maximal matchings in $G$. Then, the above question can be rephrased by asking if $\avM(G) = \frac{\totM(G)}{\maxM(G) \nM(G)}$ is close to 1. Similarly, for a parametrized family of graphs $\{G_n\}_{n\geq 0}$, one can wonder how close  $
\lim_{n \rightarrow \infty} \avM(G_n)$ is to 1. 
A graph is \emph{equimatchable} if every maximal matching is maximum. Hence, $\avM(G)=1$ for all equimatchable graphs $G$.

The graph invariants $\nM(G)$ and $\avM(G)$ are the subject of several recent studies. For example, Do\v sli\' c and Zubac~\cite{DK} show how to enumerate maximal matchings in several classes of graphs, and $\avM(G)$ could be interpreted as the expected efficiency of packings of dimers in $G$. Also, using generating functions, they compute asymptotical values of $\avM(G_n)$ for parametrized families $\{G_n\}_{n\geq 0}$ of graphs related to linear polymers. Ash and Short~\cite{Taylor}  determine the number of maximal matchings (i.e., $\nM(G)$) in three types of chemical compounds that are polyphenylene chains.  Huntemann and Neil~\cite{huntemann21} give the value of $\avM(G)$ for some types of grid graphs $G$.

A well known notion in mathematical chemistry is the Hosoya index of graph~\cite{Hosoya} which is defined as the total number of (not necessarily maximal) matchings in it. Adriantiana \emph{et al.}~\cite{Wagner} give some properties of the average size of such matchings in $G$ which we denote by $\wagner(G)$. We show in Section~\ref{sec_simdiff} that considering only \emph{maximal} matchings gives a more general approach in the sense that properties related to $\wagner(G)$ can be inferred from properties related to $\avM(G)$.  

If all maximal matchings in a graph have an equal chance to be chosen, $\avM(G)$ is the expected ratio of the size of a uniformly chosen maximal matching to the size $\maxM(G)$ of a maximum matching in $G$. However, maximal matchings are often obtained by the application of greedy algorithms and the chances of obtaining a given maximal matching by applying such a procedure are not equal. For example, the simplest heuristic  works as follows.

\begin{algorithm}[!htb]
	\caption{Randomized Greedy (RG) }
	\begin{algorithmic}[]
		\STATE set $M\leftarrow \emptyset$;
		\WHILE{$G$ contains at least one edge}
		\STATE Choose an edge $uv$ in $G$, with a uniform distribution and add it to $M$;
		\STATE Remove vertices $u$ and $v$ and all their incident edges from $G$;
		\ENDWHILE
	\end{algorithmic}
\end{algorithm}

Dyer and Frieze~\cite{DF} have analyzed the expected performance of the above algorithm. To this aim, they introduce the ratio $\dyer(G)$ of the expected size $\expM(G)$ of a randomized application of the above greedy algorithm on a graph $G$ to the maximum size $\maxM(G)$ of a matching in $G$. Note that each maximal matching $M$ produced by the RG algorithm can be obtained in $|M|!$ different ways. For example, for the path $\Path{4}$ on $4$ vertices $v_1,v_2,v_3,v_4$ and edges $v_1v_2,v_2v_3,v_3v_4$, the maximal matching $\{v_1v_2,v_3v_4\}$ can be obtained by choosing either $v_1v_2$ or $v_3v_4$ as first edge. We therefore consider every output of the RG algorithm as an {\it ordered} matching and we denote by $\Mset^o(G)$ the set of ordered maximal matchings in $G$. For illustration,  
$\Mset(\Path{4})=\{\{v_1v_2,v_3v_4\},\{v_2v_3\}\}$ while
$\Mset^o(\Path{4})=\{(v_1v_2,v_3v_4),(v_3v_4,v_1v_2),(v_2v_3)\}$.
For an ordered maximal matching  $M=(u_1v_1,\ldots,u_{|M|}v_{|M|})\in\Mset^o(G)$, let $m_i$ ($1\leq i\leq |M|$) be the number of edges in the graph obtained from $G$ by removing vertices $u_j,v_j$ ($j=1,\ldots,i-1$) and all their incident edges, and let $p(M)=\prod_{i=1}^{|M|}m_i$. Dyer and Frieze define
$$\expM(G)=\sum_{M\in \Mset^o(G)}\frac{|M|}{p(M)}\quad\mbox{and}\quad\dyer(G)=\frac{\expM(G)}{\maxM(G)}.$$

They show that there are graphs $G$ for which $\dyer(G)$ is close to $\frac{1}{2}$. They also prove that $\dyer(G)\geq \frac{6}{11}$ for all planar graphs $G$, and $\dyer(F)\geq \frac{16}{21}$ for all forests  $F$. In subsequent papers~\cite{DF93, Aronson94, Aronson95} the study of the expected performance of the RG algorithm was continued and compared to a slightly modified version where a vertex $v$ is first chosen at random, and then a random edge incident to $v$ is added to the matching. This is known as the \emph{modified randomized greedy} (MRG) algorithm. These studies have shown that MRG seems to have a better worst-case performance than RG. 
Other versions of the greedy algorithm are proposed in the literature, such as Tinhofer's {\sc MinGreedy} algorithm~\cite{Tinhofer84}. It is identical to MRG except that vertex $v$ is chosen randomly among vertices of minimum degree. An analysis of the expected performance of these greedy algorithms and of other variants can be found in \cite{goel12, besser17, poloczek12}.

Some authors have studied the difference $\maxM(G)-\expM(G)$ instead of the ratio of these two invariants. For example, Magun ~\cite{Magun98} has analyzed large random graphs having up to 10\,000 vertices, to see how many edges are lost, on average, when using a randomized greedy algorithm instead of an exact algorithm for the maximum matching problem. There exist also studies on randomized greedy matchings on particular classes of graphs such as cubic or bipartite random graphs and on weighted graphs (see e.g.,~\cite{frieze95, arnosti21, miller97} and references therein). 

The aim of this paper is to compare the average size $\frac{1}{|\Mset(G)|}\sum_{M\in \Mset(G)}|M|=\frac{\totM(G)}{\nM(G)}$ with $\maxM(G)$, independently of the application of any greedy algorithm that produces these maximal matchings. We also provide tools that help to study the asymptotic value $\lim\limits_{n\rightarrow \infty}\avM(G_n)$ for families $\{G_n\}_{n\geq 0}$ of graphs.

In the next section, we fix some notations. Then, in Section~\ref{sec_simdiff}, we investigate similarities and differences between several ratios involving matchings. For example, we highlight that $\avM(G)$ and $\dyer(G)$ do not share the same extremal properties. The main result of this paper is Theorem~\ref{thm1} presented in Section~\ref{sec_main}. It gives a simple procedure to compute the asymptotic value of $\avM(G_n)$ for many families $\{G_n\}_{n\geq 0}$ of graphs. The use of this procedure is illustrated in Section~\ref{sec_families} for families of graphs, including paths, cycles, wheels, chains of cycles, chains of cliques, ladders, and trees. 

\section{Notation} \label{sec_nota}

For basic notions of graph theory that are not defined here, we refer to Diestel~\cite{Diestel00}. Let $G = (V,E)$ be a simple undirected graph. We write $G \simeq H$ if $G$ and $H$ are two isomorphic graphs. We denote by $\K{n}$ (resp. $\Path{n}$, \C{n} and $\W{n}$) the \emph{complete graph} (resp. the \emph{path}, \emph{cycle} and the \emph{wheel}) of order $n$. We write $\K{a, b}$ for the complete bipartite graph where $a$ and $b$ are the cardinalities of the two sets of vertices of the bipartition. 
For a graph $G$ of order $n$, we define $\GT{G}$ as the graph of order $2n$ obtained by adding a new vertex $v'$ for every $v$ of $G$ and linking $v$ to $v'$. It is the \emph{thorn} graph of $G$ with all parameters equal to 1 \cite{G98}. It can also be defined as the \emph{corona product} $G \circ K_1$ of $G$ and $K_1$. For illustration, $\PathT{4}$ and $\KT{3}$ are depicted in Figure~\ref{fig:Tilde}. Note that for a graph $G$ of order $n$, there is only one maximum matching of size $n$ in $\GT{G}$ which consists in taking the edges $vv'$ for all $v$ in $G$.

\begin{figure}[!ht]
	\centering
	\scalebox{1.0}{\includegraphics{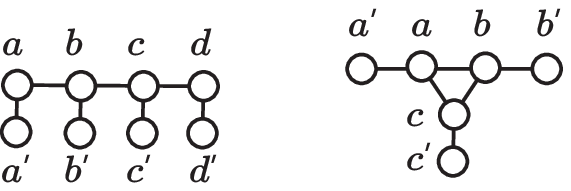}}
	\caption{$\PathT{4}$ and $\KT{3}$.}
	\label{fig:Tilde}
\end{figure}	

%The size of a maximum matching in a graph $G$ is denoted by $\maxM(G)$. 
As mentioned in the previous section, $\Mset(G)$ is the set of maximal matchings in $G$, the maximum size of a matching in $G$ is $\maxM(G)=\max_{M\in\Mset(G)}|M|$, and the total number of maximal matchings in $G$ is $\nM(G)=|\Mset(G)|$. We can compute this last value in the following alternative way: if $\Sk{G}{k}$ is the number of maximal matchings of size $k$ in $G$, then 
$$
\nM(G) = \sum_{k = 0}^{\infty} \Sk{G}{k}= \sum_{k = 0}^{m} \Sk{G}{k},
$$
where $m$ is the number of edges in $G$. Also, the sum of the sizes of all maximal matchings in $G$ is
$$
\totM(G) = \sum_{M\in\Mset(G)}|M|=\sum_{k=0}^{\infty} k \Sk{G}{k}=\sum_{k=0}^{m} k \Sk{G}{k}.
$$

We are interested in the following graph invariant $\avM(G)$ which corresponds to the ratio  of the average size of a maximal matching in a graph $G$ to the size of a maximum matching in $G$:
$$
\avM(G) = \frac{\totM(G)}{\maxM(G)\nM(G)}. 
$$ 
Observe that $\avM(G)\in \ ] \frac{1}{2}, 1]$ for all graphs $G$. Indeed, it is well known that $|M|\geq \frac{1}{2}\maxM(G)$ for all maximal matchings $M\in\Mset(G)$, and since at least one of them is maximum we get $\avM(G)>\frac{1}{2}$. It may happen that all maximal matchings are maximum, in which case $\avM(G)=1$. This is the case, for example, for cliques and complete bipartite graphs. 

As already stressed in the introduction, other similar ratios appear in the literature. For example, if $\frac{1}{p(M)}$  is the probability that the RG algorithm outputs the ordered matching $M$ from the set $\Mset^o(G)$ of ordered maximal matchings, then 
$$
\dyer(G) = \frac{1}{\maxM(G)}\sum_{M\in \Mset^o(G)}\frac{|M|}{p(M)}.
$$

Also, Adriantiana \emph{et al.}~\cite{Wagner} have studied the average size of (not necessarily maximal) matchings. By denoting $\Skw{G}{k}$  the number of matchings of size $k$ in $G$, we define $
\nMw(G) = \sum_{k = 1}^{m}\Skw{G}{k}$ and $\totMw(G) = \sum_{k=0}^{m} k \Skw{G}{k}
$
which leads to
$$
\wagner(G) = \frac{\totMw(G)}{\maxM(G)\nMw(G)}.
$$
Another possibility is to consider the average size of the ordered maximal matchings which lead to $\nMo(G) =  | \Mset^o(G) |$, $\totMo(G) =  \sum_{M \in \Mset^o(G)} | M |$ and
$$
\avMord(G) = \frac{\totMo(G)}{\maxM(G) \ \nMo(G)}.
$$
By convention, $\avM(G)=\dyer(G)=\wagner(G)=\avMord(G)=1$ for the empty graph $G$ (i.e., when $\maxM(G)=0$). In the next section, we compare these four graph invariants and their extremal properties. 
 
\section{A comparison of similar graph invariants} \label{sec_simdiff}

There is a one to one correspondence between the matchings of size $k$ in $G=(V,E)$ and the maximal matchings of size $|V|-k$ in $\GT{G}$. Indeed, consider a matching $M$ in $G$ and let $W$ be the vertices of $G$ that are not incident to an edge in $M$. A maximal matching $M'$ of $\GT{G}$ can be obtained by adding to $M$ the edges $vv'$ for the vertices $v\in W$. Since $W$ contains $|V|-2|M|$ vertices, we have $|M'|=|M|+|V|-2|M|=|V|-|M|$. Conversely, if $M'$ is a maximal matching in $\GT{G}$, let $W$ be the set of vertices of $G$ such that $vv'\in M'$. We can obtain a matching $M$ of $G$ by removing the edges $vv'$ for all $v'\in W$. We thus get a matching of size $|M|=\frac{|V|-|W|}{2}$ while $M'$ has $\frac{|V|-|W|}{2}+|W|=|V|-(\frac{|V|-|W|}{2})=|V|-|M|$ edges. 

The above relation between the matchings in $G$ and the maximal matchings in $\GT{G}$ implies
$\Sk{\GT{G}}{|V|-k}=\Skw{G}{k}$,
$\nMw(G)=\nM(\GT{G})$
%Therefore,
%$$\nMw(G)=\sum_{k=0}^{\infty} \Skw{G}{k}=\sum_{k=0}^{\infty}  \Sk{\GT{G}}{|V|-k}=\nM(\GT{G})$$
and
\begin{align*}
\totMw(G)=&\sum_{k=0}^{\infty} k \Skw{G}{k}=
\sum_{k=0}^{\infty} k \Sk{\GT{G}}{|V|-k}\\
=&
\sum_{k=0}^{\infty} |V| \Sk{\GT{G}}{|V|-k}-
\sum_{k=0}^{\infty} (|V|-k) \Sk{\GT{G}}{|V|-k}\\
=&
|V|\nM(\GT{G})-\totM(\GT{G}).
\end{align*}
Since $\maxM(\GT{G})=|V|$, this implies
\begin{align*}
&\wagner(G)=\frac{\totMw(G)}{\maxM(G)\nMw(G)}=
\frac{|V|\nM(\GT{G})-\totM(\GT{G})}{\maxM(G)\nM(\GT{G})}=
\frac{|V|(1-\avM(\GT{G}))}{\maxM(G)}\\
\Leftrightarrow\quad&\avM(\GT{G})=1-\frac{\maxM(G)\wagner(G)}{|V|}.
\end{align*}

For illustration, consider the clique $\K{2}$ on two vertices $a,b$. There are 2 matchings in $\K{2}$, namely the empty set of size 0, and $\{ab\}$ of size 1. Since $\maxM(\K{2})=1$ we get $\wagner(\K{2})=\frac{1}{2}$. Note that $\KT{2}\simeq \Path{4}$, which implies 
$\avM(\Path{4})=1-\frac{1}{2}\frac{1}{2}=\frac{3}{4}$. 
We could have computed $\wagner(\K{2})$ from $\avM(\Path{4})$ by observing that the two maximal matchings in $\Path{4}$ are $\{aa',bb'\}$ and $\{ab\}$, which gives
$\avM(\Path{4})=\frac{3}{4}$, and which implies 
$\wagner(\K{2})=2(1-\frac{3}{4})=\frac{1}{2}$.
Also, it is proven in \cite{Wagner} that $\lim\limits_{n\rightarrow \infty} \maxM(\K{n})\wagner(\K{n})=\frac{n}{2}$. The above link between $\wagner(G)$ and $\avM(\GT{G})$ implies 
$$\lim_{n\rightarrow \infty} \avM(\KT{n})\sim 1-\frac{\frac{n}{2}}{n}=\frac{1}{2}.$$

Note that while $\wagner(G)$ can be derived from $\avM(\GT{G})$, the converse is not always true since $\avM(G)$ can be derived from $\wagner(H)$ only if there is such a graph $H$ with $G\simeq \GT{H}$.

We are now comparing $\avM(G)$, $\avMord(G)$ and $\dyer(G)$ which are all a ratio between the average size of maximal matchings and the size of a maximum matching. For illustration, consider the graph $\KT{3}$ of Figure \ref{fig:Tilde}:
\begin{itemize} 
	\item there are four maximal matchings, namely $\{aa',bb',cc'\}$ of size 3 and $\{aa',bc\}$, $\{bb',ac\}$, $\{cc',ab\}$ of size 2, and since $\maxM(\KT{3})=3$, we get $\avM(\KT{3})=\frac{9}{12}=\frac{3}{4}$;
	\item the 12 ordered maximal matchings are $(aa',bb',cc')$, $(aa',cc',bb')$, $(bb',aa',cc')$ $(bb',cc',aa')$, $(cc',aa',bb')$, $(cc',bb',aa')$, $(aa',bc)$, $(bc,aa')$, $(bb',ac)$, $(ac,bb')$, $(cc',ab)$ and $(ab,cc')$. Hence, six of them are of size 3 and the six others are of size 2, which gives $\avMord(\KT{3})=\frac{30}{36}=\frac{5}{6}$;
	\item the probability that the output of the RG algorithm is $(aa',bb',cc')$ is $\frac{1}{6}\frac{1}{3}=\frac{1}{18}$. It is the same for the five other ordered maximal matchings of size 3. The probability that RG produces $(aa',bc)$ is also $\frac{1}{6}\frac{1}{3}=\frac{1}{18}$, and it is the same for $(bb',ac)$ and $(cc',ab)$. The probability that RG produces $(bc,aa')$ is $\frac{1}{6}\frac{1}{1}=\frac{1}{6}$ and it is the same for $(ac,bb')$ and $(ab,cc')$. Hence, $\dyer(\KT{3})=\frac{1}{3}(6\frac{3}{18}+3\frac{2}{18}+3\frac{2}{6})=\frac{7}{9}$.
	\end{itemize}{}
\begin{figure}[!ht]
	\centering
	\scalebox{0.23}{\includegraphics{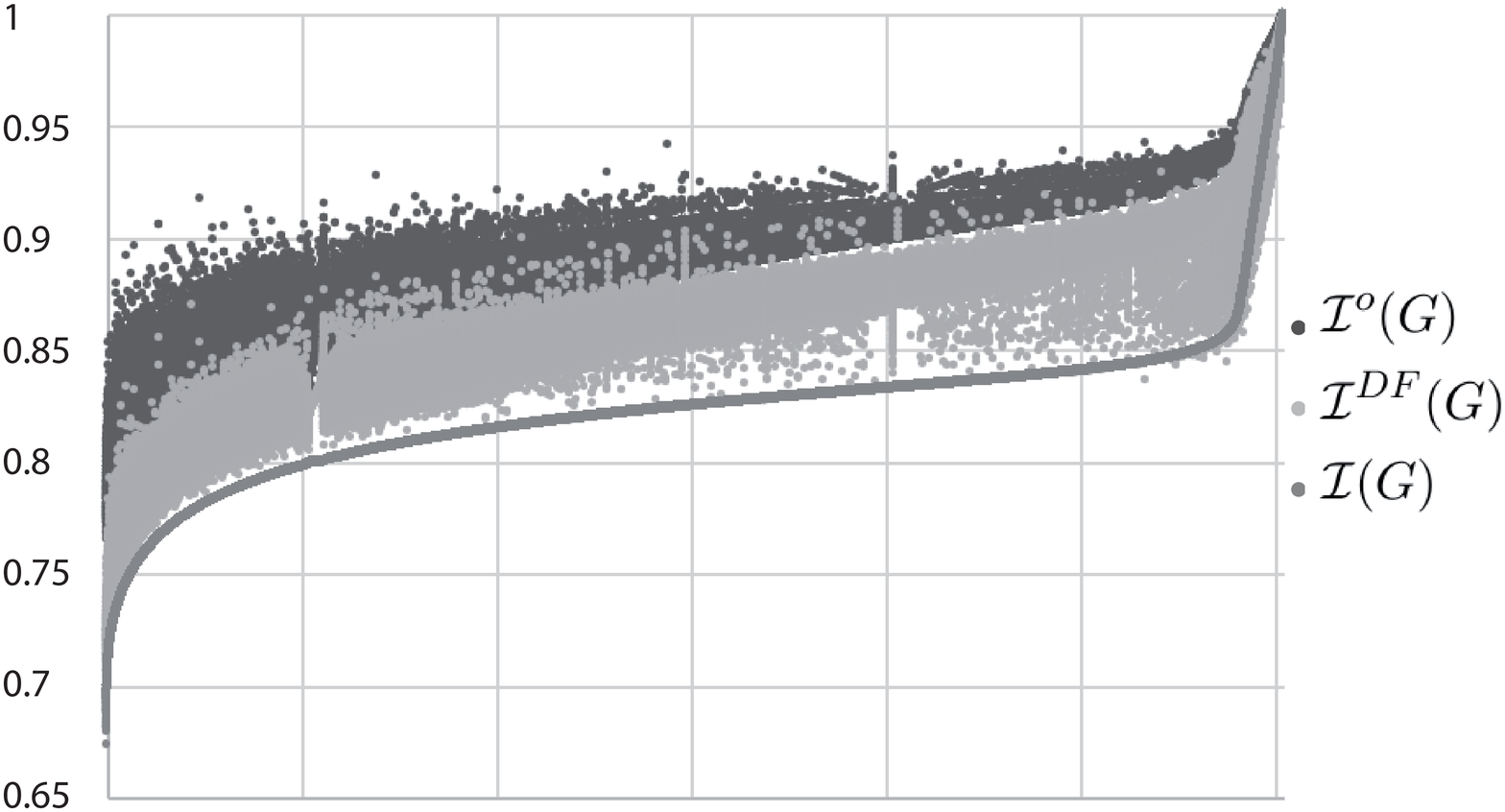}}
	\caption{A comparison of $\avM(G)$, $\avMord(G)$ and $\dyer(G)$ on all graphs of order 10.}
	\label{fig:Comparison}
\end{figure}
The above example  illustrates well the fact that these three invariants differ from each other. To better observe these differences we have generated all non isomorphic graphs of order $10$ and computed the values of the graph invariants
using \PHOEG~\cite{PHOEG}.
There are 12\,005\,168 such graphs, and we have sorted them by non decreasing $\avM(G)$ value. The values of $\avM(G)$, $\avMord(G)$ and $\dyer(G)$ are shown in Figure \ref{fig:Comparison}. We observe that $\avMord(G)\geq \avM(G)$ for all $G$. This is not a surprise since there are exactly $k!\Sk{G}{k}$ ordered maximal matchings of size $k$ in $G$, which implies
\begin{align*}
\maxM(G)(\avMord(G)-\avM(G))=&\frac{\displaystyle\sum_{k=0}^{\infty}kk!\Sk{G}{k}}{\displaystyle\sum_{k=0}^{\infty}k!\Sk{G}{k}}-\frac{\displaystyle\sum_{k=0}^{\infty}k\Sk{G}{k}}{\displaystyle\sum_{k=0}^{\infty}\Sk{G}{k}}\\=&\frac{\displaystyle\sum_{k=0}^{\infty}\displaystyle\sum_{k'\geq k}\Sk{G}{k}\Sk{G}{k'}\Big((kk!+k'k'!)-(kk'!-k'k!)\Big)}
{\displaystyle\sum_{k=0}^{\infty}k!\Sk{G}{k}\displaystyle\sum_{k=0}^{\infty}\Sk{G}{k}}\\
=&\frac{\displaystyle\sum_{k=0}^{\infty}\displaystyle\sum_{k'\geq k}\Sk{G}{k}\Sk{G}{k'}(k'-k)(k'!-k!)}
{\displaystyle\sum_{k=0}^{\infty}k!\Sk{G}{k}\displaystyle\sum_{k=0}^{\infty}\Sk{G}{k}}\geq 0.
\end{align*}

When comparing $\dyer(G)$ with the two other graph invariants, we observe that $\dyer(G)$ is typically larger than $\avM(G)$ and typically smaller than $\avMord(G)$. But there are some exceptions. We have found that $\dyer(G)>\avMord(G)$ for 18 graphs of order 10 and $\dyer(G)<\avM(G)$ for 4\,359 of them. In Figure \ref{fig:exceptions}, we illustrate examples of such exceptions for graphs of order 6. It is not difficult to check that $\avM(G_1)=\frac{19}{20}>\frac{15}{16}=\dyer(G_1)$ and
$\avMord(G_2)=\frac{32}{39}<\frac{33}{40}=\dyer(G_2)$. Note that the observed difference between $\avM$ and $\dyer$  dates back to the 1930s, as illustrated by the works of Flory \cite{Flory} and Jackson and Montroll \cite{JM}. 

\begin{figure}[!ht]
	\centering
	\scalebox{0.85}{\includegraphics{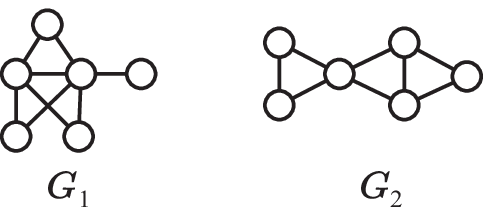}}
	\caption{Two graphs $G_1$ and $G_2$ with $\avM(G_1)>\dyer(G_1)$ and $\avMord(G_2)<\dyer(G_2)$.}
	\label{fig:exceptions}
\end{figure}

Andriantiana {\it et al.}\cite{Wagner} have proven that the trees $T$ of order $n$ that maximize the average size $\wagner(T)\maxM(T)$ of a (not necessarily maximal) matching are the paths $\Path{n}$. It follows from the above equations that $\avM(\PathT{n})\leq \avM(\GT{T})$ for all trees $T$ of order $n$. Note that $\PathT{n}$ is a tree on $2n$ vertices.
Dyer and Frieze \cite{DF} conjecture that 
$\dyer(\PathT{n})\leq \dyer(T)$ for all trees $T$ of order $2n$.
We show in Section \ref{sec:trees} that there are trees $T$ of order $2n$ such that $\avM(\PathT{n})> \avM(T)$. Hence, $\avM$ and $\dyer$ do not seem to share the same extremal properties.

\section{Main result} \label{sec_main}

The purpose of this section is to provide a general technique for determining $\lim_{n\rightarrow \infty}\avM(G_n)$ for classes $\{G_n\}_{n\geq 0}$ of graphs where the number of maximal matchings in $G_n$ linearly depends on the number of maximal matchings in graphs $G_{n'}$ with $n'<n$.
For this purpose, Let us first prove a result which comes from the theory which relates sequences satisfying linear equations to bivariate generating functions (see for example Theorems IV.9 and IV.10 in \cite{FS09}).

\begin{thm}\label{thm1}
	Let $(f_{n,k})$ be a sequence of numbers depending on two positive integer-valued
	indices, $n$ and $k$, and such that 
	$$f_{n,k}=\sum_{i=1}^I\sum_{j=1}^{n_i}a_{ij}f_{n-i,k-j}$$
	for a strictly positive integer $I$, non-negative integers $n_i$ ($i=1,\ldots,I$) and real numbers $a_{ij}$ ($i=1,\ldots,I; j=1,\ldots,n_i$).
	Let $$F(x,y)=\frac{P(x,y)}{Q(x,y)}=\sum_{n\geq 0}\sum_{k\geq 0}f_{n,k}x^ny^k$$ be the associated bivariate generating function with $Q(x,y)=1-\sum\limits_{i=1}^I\sum\limits_{j=1}^{n_i}a_{ij}x^iy^j$.
	
If $Q(x,1)$ has a unique root $\rho$ of smallest modulus with multiplicity 1, and if $P(\rho,1)\neq 0$, then 
$$\lim\limits_{n\rightarrow \infty}\; \frac{\sum\limits_{k\geq 0}kf_{n,k}}{n\sum\limits_{k\geq 0}f_{n,k}}=\frac{\sum\limits_{i=1}^I\sum\limits_{j=1}^{n_i}ja_{ij}\rho^i}{\sum\limits_{i=1}^I\sum\limits_{j=1}^{n_i}ia_{ij}\rho^i}.$$
\end{thm}
\begin{proof}
	Note first that the denominator of $F(x,1)$ can be factored as $Q(x,1)= (1-x/\rho)q(x)$,
	where $q$ only has roots whose modulus is greater than that of $\rho$. Dividing by $x-\rho$ and taking
	the limit $x \rightarrow \rho$, we obtain 
	$\frac{\partial Q}{\partial x}(\rho,1)=-q(\rho)/\rho$ or equivalently $q(\rho)=\rho\frac{\partial Q}{\partial x}(\rho,1)$.
	
	Since we assume $P(\rho,1)\neq 0$, $F(x,1)=\frac{P(x,1)}{Q(x,1)}$ has a simple pole at $\rho$ and no other pole whose modulus is less than or equal to that of $\rho$. At the pole, we have
	$$F(x,1)=\frac{P(x,1)}{(1-x/\rho)q(x)}\sim \frac{P(\rho,1)}{q(\rho)}(1-x/\rho)^{-1}.$$
	By partial faction expansion, one can express $F(x,1)$ as
	$$F(x,1)=\frac{P(\rho,1)}{q(\rho)}(1-x/\rho)^{-1}+\tilde{F}(x),$$
	where $\tilde{F}$ is holomorphic in a disc around $0$ whose radius is greater than $|\rho|$. Applying singularity analysis in the rational/meromorphic case, it follows that
	$$[x^n]F(x,1)=\sum_{k\geq 0}f_{n,k}\sim\frac{P(\rho,1)}{q(\rho)}\rho^{-n}=-\frac{P(\rho,1)}{\rho\frac{\partial Q}{\partial x}(\rho,1)}\rho^{-n}.$$
	Likewise,
	$$\frac{\partial F}{\partial y}(x,1)=\frac{Q(x,1)\frac{\partial P}{\partial y}(x,1)-P(x,1)\frac{\partial Q}{\partial y}(x,1)}{Q(x,1)^2}$$
	has a pole (of multiplicity 2) at $\rho$, but no other poles in a disc around 0 whose radius is greater than $\rho$. At this pose, we have
	$$Q(x,1)\frac{\partial P}{\partial y}(x,1)-P(x,1)\frac{\partial Q}{\partial y}(x,1)=-P(\rho,1)\frac{\partial Q}{\partial y}(\rho,1)+O(x-\rho)$$
	and
	$$Q(x,1)^2=-(q(\rho)^2+O(x-\rho))(1-x/\rho)^2,$$
	thus 
	$$\frac{\partial F}{\partial y}(x,1)=-\frac{P(\rho,1)\frac{\partial Q}{\partial y}(\rho,1)}{q(\rho)^2}(1-x/\rho)^2+O((x-\rho)^{-1}).$$
	Applying singularity analysis now yields
	$$[x^n]\frac{\partial F}{\partial y}(x,1)=\sum_{k\geq 0}kf_{n,k}=\Bigl(\frac{-P(\rho,1)\frac{\partial Q}{\partial y}(\rho,1)}{q(\rho)^2}n+O(1)\Bigr)\rho^{-n}
	=\Bigl(\frac{-P(\rho,1)\frac{\partial Q}{\partial y}(\rho,1)}{\rho^2\frac{\partial Q}{\partial x}(\rho,1)^2}n+O(1)\Bigr)\rho^{-n}.$$
	Taking the quotient of the two asymptotic formulas, we obtain 
	$$\lim\limits_{n\rightarrow \infty}\; \frac{\sum\limits_{k\geq 0}kf_{n,k}}{n\sum\limits_{k\geq 0}f_{n,k}}=
	\frac{\frac{\partial Q}{\partial y}{(\rho,1)}}{\rho\frac{\partial Q}{\partial x}{(\rho,1)}}=
	\frac{\sum\limits_{i=1}^I\sum\limits_{j=1}^{n_i}ja_{ij}\rho^i}{\sum\limits_{i=1}^I\sum\limits_{j=1}^{n_i}ia_{ij}\rho^i}.$$
\end{proof}
\begin{cor}\label{cor1}
Let $\{G_n\}_{n\geq 0}$ be a family of graphs such that 
$$\Sk{G_n}{k}=\sum_{i=1}^I\sum_{j=1}^{n_i}a_{ij}\Sk{G_{n-i}}{k-j}$$
for a strictly positive integer $I$, non-negative integers $n_i$ ($i=1,\ldots,I$) and real numbers $a_{ij}$ ($i=1,\ldots,I; j=1,\ldots,n_i$).
Let $\alpha_i=\sum_{j=1}^{n_i}a_{ij}$ and $\beta_i=\sum_{j=1}^{n_i}ja_{ij}$, and let $c=\lim\limits_{n\rightarrow \infty}\frac{\maxM(G_n)}{n}$. 
If the equation $x^I-\sum_{i=1}^{I}\alpha_{i}x^{I-i}=0$ has a unique root $r$ of maximum modulus with multiplicity 1 and if 
$$\sum_{i=0}^{I-1}\frac{1}{r^i}\Bigl(T_0(G_i)-\sum_{j=1}^i\alpha_jT_0(G_{i-j})\Bigr)\neq 0$$
then
$$\lim\limits_{n\rightarrow \infty}\avM(G_n)= \frac{\sum_{i=1}^I\beta_ir^{I-i}}{c\sum_{i=1}^Ii\alpha_ir^{I-i}}.$$	
\end{cor}
\begin{proof}
	Let $f_{n,k}=S(G_n,k)$ and let $$F(x,y)=\frac{P(x,y)}{Q(x,y)}=\sum_{n\geq 0}\sum_{k\geq 0}f_{n,k}x^ny^k$$ be the associated bivariate generating function with $Q(x,y)=1-\sum_{i=1}^I\sum_{j=1}^{n_i}a_{ij}x^iy^j$. Note that 
	$$Q(x,1)=1-\sum_{i=1}^I\sum_{j=1}^{n_i}a_{ij}x^i=1-\sum_{i=1}^I\alpha_ix^i.$$
	Hence, assuming that  $x^I-\sum_{i=1}^{I}\alpha_{i}x^{I-i}=0$ has a unique root $r$ of maximum modulus and of multiplicity 1 is equivalent to assuming that  $Q(x,1)$ has a unique root $\frac{1}{r}$ of smallest modulus with multiplicity 1.
	It is not difficult to check that $$P(x,1)=\sum_{i=0}^{I-1}x^i\Bigl(T_0(G_i)-\sum_{j=1}^i\alpha_jT_0(G_{i-j})\Bigr).$$
	Since we assume $P(\frac{1}{r},1)\neq 0$, we know from Theorem \ref{thm1} that 
	$$\lim\limits_{n\rightarrow \infty}\avM(G_n)=
	\lim\limits_{n\rightarrow \infty}\; \frac{\sum\limits_{k\geq 0}kf_{n,k}}{cn\sum\limits_{k\geq 0}f_{n,k}}=
	\frac{\sum\limits_{i=1}^I\sum\limits_{j=1}^{n_i}ja_{ij}r^{-i}}{c\sum\limits_{i=1}^I\sum\limits_{j=1}^{n_i}ia_{ij}r^{-i}}=
	\frac{\sum\limits_{i=1}^I\beta_ir^{I-i}}{c\sum\limits_{i=1}^Ii\alpha_ir^{I-i}}.$$
\end{proof}

\section{Applications on some families of graphs} \label{sec_families}
In this section we illustrate how Corollary \ref{cor1} allows to determine the asymptotic behavior of the invariant $\avM$ for various families of graphs. These families were chosen so as to cover a wide range of values in the interval  $] \frac{1}{2}, 1].$
\subsection{Paths, cycles and wheels}

The plastic number $r$  is the root of maximum modulus and the unique real solution of the cubic equation $x^3=x+1$ \cite{plastic}. Its value is 
$r=\sqrt[3]{\frac{9+\sqrt{69}}{18}}+\sqrt[3]{\frac{9-\sqrt{69}}{18}}$.

\begin{thm}\label{thm:path}
$\lim\limits_{n\rightarrow \infty}\avM(\Path{n})=\frac{2r+2}{2r+3}$ where $r$ is the plastic number.
\end{thm}
\begin{proof}
	Let $\Path{n}$ be a path on $n$ vertices $v_1,\ldots,v_n$ with edges $v_iv_{i+1}$ ($i=1,\ldots,n-1$), and let $M$ be a maximal matching of size $k$ in $\Path{n}$. If $v_1v_2\in M$, then $M\setminus \{v_1v_2\}$ is a maximal matching in $\Path{n-2}$ obtained from $\Path{n}$ by removing $v_1$ and $v_2$. If $v_1v_2\notin M$, then $v_2v_3\in M$, which means that $M\setminus \{v_2v_3\}$ is a maximal matching in $\Path{n-3}$ obtained from $\Path{n}$ by removing $v_1$, $v_2$ and $v_3$. We therefore have
	$$\Sk{\Path{n}}{k}=\Sk{\Path{n-2}}{k-1}+\Sk{\Path{n-3}}{k-1}.$$
	With the notations of Corollary \ref{cor1}, we have $I=3,  \alpha_1=\beta_1=0$ and $ \alpha_2=\alpha_3=\beta_2=\beta_3=1$. The root of maximum modulus of the equation $x^3-x-1=0$ is the plastic number $r$. Moreover, since $T_0(\Path{0})=T_0(\Path{1})=T_0(\Path{2})=1$, we have
	$$\sum_{i=0}^{I-1}\frac{1}{r^i}\Bigl(T_0(\Path{i})-\sum_{j=1}^i\alpha_jT_0(\Path{i-j})\Bigr)=1+\frac{1}{r}\neq 0.$$
	Hence the hypotheses of Corollary \ref{cor1} are satisfied, and since 
 $\maxM(\Path{n})=\lfloor\frac{n}{2}\rfloor$, we have $\lim\limits_{n\rightarrow \infty}\frac{\maxM(\Path{n})}{n}=\frac{1}{2}$ which implies
	\begin{align*}
\lim\limits_{n\rightarrow \infty}\avM(\Path{n})=&\frac{r+1}{\frac{1}{2}(2r+3)}=\frac{2r+2}{2r+3}\approx 0.8299.&&\qedhere
\end{align*}
\end{proof}
\begin{cor}
$\lim\limits_{n\rightarrow \infty}\avM(\C{n})=\lim\limits_{n\rightarrow \infty}\avM(\Path{n}).$
\end{cor}
\begin{proof}
	Let $\C{n}$ be a cycle on $n$ vertices $v_1,\ldots,v_n$ with edges $v_iv_{i+1}$ ($i=1,\ldots,n{-}1$) and $v_1v_n$, and let $M$ be a maximal matching of size $k$ in $\C{n}$. If $v_1v_2\in M$, then $M\setminus \{v_1v_2\}$ is a maximal matching  of $\Path{n-2}$ obtained from $\C{n}$ by removing $v_1$ and $v_2$. If $v_1v_2\notin M$, then $M$ contain at least one of the edges $v_2v_3$ and $v_1v_n$, and if $M$ contains both of them, then $M\setminus \{v_2v_3,v_1v_n\}$ is a maximal matching in $\Path{n-4}$ obtained from $\C{n}$ by removing $v_1$, $v_2$, $v_3$ and $v_4$. Hence,
\begin{eqnarray}
\Sk{\C{n}}{k}&=&\Sk{\Path{n-2}}{k-1}+\Big(2\Sk{\Path{n-2}}{k-1}-\Sk{\Path{n-4}}{k-2}\Big)\nonumber\\
&=&3\Sk{\Path{n-2}}{k-1}-\Sk{\Path{n-4}}{k-2}\nonumber.
\end{eqnarray}
We know from Theorem \ref{thm:path} that $\Sk{\Path{n}}{k}=\Sk{\Path{n-2}}{k-1}+\Sk{\Path{n-3}}{k-1}$, which implies
\begin{eqnarray}
\Sk{\C{n}}{k}&=&3\Big(\Sk{\Path{n-4}}{k-2}+\Sk{\Path{n-5}}{k-2}\Big)-\Big(\Sk{\Path{n-6}}{k-3}+\Sk{\Path{n-7}}{k-3}\Big)\nonumber.\\
&=&\Sk{\C{n-2}}{k-1}+\Sk{\C{n-3}}{k-1}\nonumber.
\end{eqnarray}
This is the same recurrence relation as in Theorem \ref{thm:path}. Hence, $T_0(\C{n})=T_0(\C{n-2})+T_0(\C{n-3})$. We can fix $T_0(\C{0})=3$, $T_0(\C{1})=0$ and $T_0(\C{2})=2$ (even if $\C{0},\C{1}$ and $\C{2}$ do not exist), so that the values of $T_0(\C{n}))$ are the correct ones for $n\geq 3$. Hence,
	$$\sum_{i=0}^{I-1}\frac{1}{r^i}\Bigl(T_0(\C{i})-\sum_{j=1}^i\alpha_jT_0(\C{i-j})\Bigr)=3-\frac{1}{r^2}\neq 0,$$
	where $r$ is the plastic number. Since $\maxM(\C{n}){=}\maxM(\Path{n})$, we conclude that  
$\lim\limits_{n\rightarrow \infty}\avM(\C{n}){=}\lim\limits_{n\rightarrow \infty}\avM(\Path{n}).$
\end{proof}

\begin{cor}
	$\lim\limits_{n\rightarrow \infty}\avM(\W{n})=\lim\limits_{n\rightarrow \infty}\avM(\Path{n}).$
\end{cor}
\begin{proof}
	Let $\W{n}$ be a wheel on $n$ vertices $v_1,\ldots,v_n$ where $v_1$ is the center of the weel. If $n$ is even, then all maximal matchings in $\W{n}$ contain an edge incident to $v_1$, while if $n$ is odd, then there are exactly two maximal matchings with no edge incident to $v_1$, and both of them contain exactly $\frac{n-1}{2}$ edges. Since the graph obtained from $\W{n}$ by removing the endpoints of an edge incident to $v_1$ is a path on $n-2$ vertices, we have 		
	$$
	\Sk{\W{n}}{k}=(n-1)\Sk{\Path{n-2}}{k-1}+
	\begin{cases}
	2&\mbox{if }n \mbox{ is odd and }k=\frac{n-1}{2},\\
	0&\mbox{otherwise.}
	\end{cases}
$$
It follows that 
$$\nM(\W{n})=(n-1)\nM(\Path{n-2})+
\begin{cases}
2&\mbox{if }n \mbox{ is odd,}\\
0&\mbox{if }n \mbox{ is even.}
\end{cases}
$$
and
$$\totM(\W{n})=(n-1)(\totM(\Path{n-2})+\nM(\Path{n-2}))+
\begin{cases}
n-1&\mbox{if }n \mbox{ is odd,}\\
0&\mbox{if }n \mbox{ is even.}
\end{cases}
$$
Hence,  
$$\frac{\totM(\W{n})}{\nM(\W{n})}
=\frac{\totM(\Path{n-2})+\nM(\Path{n-2})+1(n\;\mathrm{mod}\;2)}{\nM(\Path{n-2})+\frac{2}{n-1}(n\;\mathrm{mod}\;2)}.$$
Since $\maxM(\Path{n-2})=\maxM(\W{n})-1$, we have
\begin{align}
\lim\limits_{n\rightarrow \infty}\avM(\W{n})=&
\lim\limits_{n\rightarrow \infty}\frac{\totM(\Path{n-2})+\nM(\Path{n-2})}{\maxM(\W{n})\nM(\Path{n-2})}&&\nonumber\\
=&\lim\limits_{n\rightarrow \infty}\Big(\frac{(\maxM(\W{n})-1)\avM(\Path{n-2})}{\maxM(\W{n})}+\frac{1}{\maxM(\W{n})}\Big)&&\nonumber\\
=&\lim\limits_{n\rightarrow \infty}\Big(\avM(\Path{n-2})+\frac{1-\avM(\Path{n-2})}{\maxM(\W{n})}\Big)&\nonumber\\
=&\lim\limits_{n\rightarrow \infty}\avM(\Path{n}).&&\nonumber\qedhere
\end{align}
\end{proof}

\subsection{Chains of cycles}\label{sec_chains_cycle}
The types of graphs analyzed in the following subsections are numerous, and to avoid giving a specific name to each of them, we will use the notation $G_n^i$ for the $i$th type of graph, possibly with additional information next to the exponent $i$. Also, while we will always give the root $r$ of maximum modulus mentioned in Corollary \ref{cor1}, we leave to the reader the task to check that $\sum_{i=0}^{I-1}\frac{1}{r^i}\Bigl(T_0(G_i)-\sum_{j=1}^i\alpha_jT_0(G_{i-j})\Bigr)\neq 0$.

We first consider chains of hexagons studied in \cite{Taylor}, where the cut-vertices of the hexagons are at distance 3, 2 or 1. More precisely, let $G_{n}^{1,s}$ be the graph obtained by considering the disjoint union of $n$ hexagons $H_1,\ldots,H_n$, were $\{v_{i}^1,\ldots,v_{i}^6\}$ and $\{v_{i}^1v_{i}^2,\ldots,v_{i}^5v_{i}^6,v_{i}^1v_{i}^6\}$ are the vertex set and the edge set of $H_i$, and by adding an edge between $v_{i}^{s+1}$ and $v_{i+1}^1$ for $i=1,\ldots,n-1$. For illustration $G_{3}^{1,1}$, $G_{3}^{1,2}$ and $G_{3}^{1,3}$ are depicted in Figure \ref{fig:Hexagons}. They are known as  the ortho-, the meta-, and the para- phenylene chains, respectively \cite{Taylor}. Maximal matchings in analogous types of spiro-chains were considered in \cite{DS21}.

\begin{figure}[!ht]
	\centering
	\scalebox{1.0}{\includegraphics{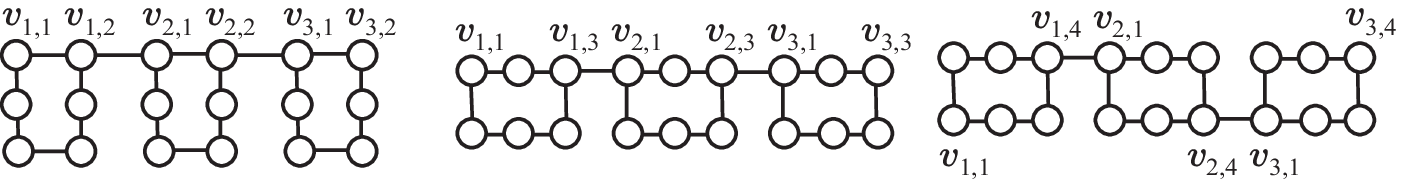}}
	\caption{$G_{3}^{1,1}$, $G_{3}^{1,2}$ and $G_{3}^{1,3}$.}
	\label{fig:Hexagons}
\end{figure}

Approximate values for $\lim\limits_{n\rightarrow \infty}\avM(G_{n}^{1,s})$ are given in \cite{Taylor}. We show how to compute these values using the technique of Theorem \ref{thm1}. Since the authors made calculation errors for $s=2$, we give a detailed proof for this case, while less details are given for $s=1$ and $s=3$. So let's start with $s=2$ and let $G_{n}^{1,2,1}$, $G_{n}^{1,2,2}$ and $G_{n}^{1,2,3}$ be defined as follows:
	\begin{itemize}\setlength\itemsep{0.9pt}
		\item 
		$G_{n}^{1,2,1}$ is the graph obtained from  $G_{n}^{1,2}$ by adding a vertex $w_1$ linked to $v_1^{1}$, 
		\item 
		$G_{n}^{1,2,2}$ is the graph obtained from  $G_{n}^{1,2,1}$ by adding a vertex $w_2$ linked to $w_1$,
		\item 
		$G_{n}^{1,2,3}$ is the graph obtained from  $G_{n}^{1,2,2}$ by adding a vertex $w_3$ linked to $w_1$, 
	\end{itemize}
For illustration, $G_{2}^{1,2,1}$, $G_{2}^{1,2,2}$ and $G_{2}^{1,2,3}$ are depicted in Figure \ref{fig:Ans}.

\begin{figure}[!ht]
	\centering
	\scalebox{1.0}{\includegraphics{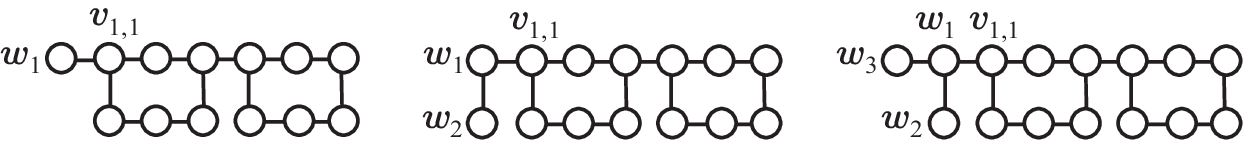}}
	\caption{$G_{n}^{1,2,1}$, $G_{n}^{1,2,2}$ and $G_{n}^{1,2,3}$.}
	\label{fig:Ans}
\end{figure}	

\begin{thm}\label{thm:hexa1}
	$\lim\limits_{n\rightarrow \infty}\avM(G_{n}^{1,2})=\frac{16r^2-3r+12}{3(7r^2-4r+6)}$  where $r=\frac{7+\sqrt[3]{307+9\sqrt{182}}+\sqrt[3]{307-9\sqrt{182}}}{3}.$
\end{thm}
\begin{proof}
	To simplify the notation, let's write $S(n,k)$ and $S_j(n,k)$ instead of $\Sk{G_{n}^{1,2}}{k}$ and $\Sk{G_{n}^{1,2,j}}{k}$. As mentioned  in \cite{Taylor},
	\begin{align}
	S(n,k)&=2S(n-1,k-2)+S_1(n-1,k-2)+2S_2(n-1,k-2) \label{eq:ans1}\\[0.5em]
	S_1(n,k)&=S(n-1,k-2)+S_1(n-1,k-2)+3S_2(n-1,k-2)+S_3(n-1,k-2)\label{eq:ans2}\\[0.5em]
	S_2(n,k)&=S(n,k-1)+S_2(n-1,k-2)+S_3(n-1,k-2)\label{eq:ans3}\\[0.5em]
	S_3(n,k)&=2S(n,k-1)+S_2(n-1,k-2)+S_3(n-1,k-2).\label{eq:ans4}
	\end{align}
	
	Equations (\ref{eq:ans3}) and (\ref{eq:ans4}) show that $S_3(n,k)=S_2(n,k)+S(n,k-1)$ which implies that (\ref{eq:ans4}) can be rewritten as
	\begin{align}
	&&	S_2(n,k)+S(n,k{-}1)=&2S(n,k{-}1)+S_2(n{-}1,k{-}2)+(S_2(n{-}1,k{-}2)+S(n{-}1,k{-}3))\nonumber\\[0.5em]
	\Leftrightarrow&& S_2(n,k)=&S(n,k-1)+S(n-1,k-3)+2S_2(n-1,k-2).\tag{\ref{eq:ans4}'}
	\end{align}
	Hence, Equation (\ref{eq:ans2}) can be rewritten as
	\begin{align}
	&&	S_1(n,k)=&S(n-1,k-2)+\Big(S(n,k)-2S(n{-}1,k{-}2)-2S_2(n-1,k-2)\Big)\nonumber\\
	&&&+3S_2(n-1,k-2)+\Big(S_2(n-1,k-2)+S(n-1,k-3)\Big)\nonumber\\[0.5em]
	&&	=&S(n,k)-S(n{-}1,k{-}2)+S(n-1,k-3)+2S_2(n-1,k-2)\nonumber\\[0.5em]
	&&	=&S(n,k)-S(n{-}1,k{-}2)+S(n{-}1,k{-}3)+(S_2(n,k)-S(n,k{-}1)-S(n{-}1,k{-}3))\nonumber\\[0.5em]
	&&	=&S(n,k)-S(n,k{-}1)-S(n{-}1,k{-}2)+S_2(n,k)\nonumber\tag{\ref{eq:ans2}'}
	\end{align}
	which implies that Equation (\ref{eq:ans1}) can be rewritten as
	\begin{align}
	&	S(n,k)&=& 2S(n{-}1,k{-}2)+\Big(S(n{-}1,k{-}2)-S(n{-}1,k{-}3)-S(n{-}2,k{-}4)+S_2(n{-}1,k{-}2)\Big)\nonumber\\
	&&&+2S_2(n{-}1,k{-}2)\nonumber\\
	&&=&3S(n-1,k-2)-S(n{-}1,k{-}3)-S(n{-}2,k{-}4)+3S_2(n{-}1,k{-}2)\nonumber
    \end{align}
	\begin{align}
	\Leftrightarrow
	3S_2(n{-}1,k{-}2)=S(n,k)-3S(n-1,k-2)+S(n{-}1,k{-}3)+S(n{-}2,k{-}4).\tag{\ref{eq:ans1}'}
	\end{align}
	Using Equation (\ref{eq:ans1}'), we can rewrite Equation (\ref{eq:ans4}') as follows:
	\begin{align*}
	&S(n{+}1,k{+}2)-3S(n,k)+S(n,k{-}1)+S(n{-}1,k{-}2)\nonumber\\
	&\quad=3S(n,k{-}1){+}3S(n{-}1,k{-}3){+}\Big(2S(n,k){-}6S(n{-}1,k{-}2){+}2S(n{-}1,k{-}3){+}2S(n{-}2,k{-}4)\Big)\nonumber\\
	\Leftrightarrow\quad&S(n{+}1,k{+}2)=5S(n,k)+2S(n,k{-}1)-7S(n{-}1,k{-}2)+5S(n{-}1,k{-}3)+2S(n{-}2,k{-}4)\nonumber\\
	\Leftrightarrow\quad&S(n,k)=5S(n{-}1,k{-}2)+2S(n{-}1,k{-}3)-7S(n{-}2,k{-}4)+5S(n{-}2,k{-}5)+2S(n{-}3,k{-}6).
	\end{align*}
	We have $I=3$, $\alpha_1=7$, $\alpha_2=-2$, $\alpha_3=2$, $\beta_1=16$, $\beta_2=-3$ and $\beta_3=12$, and the root of maximum modulus of the equation $x^3-7x^2+2x-2=0$ is
	$$r=\frac{7+\sqrt[3]{307+9\sqrt{182}}+\sqrt[3]{307-9\sqrt{182}}}{3}.$$ Since $\frac{\maxM(G_{n}^{1,2})}{n}=3$, it follows that
\begin{align}\lim\limits_{n\rightarrow \infty}\avM(G_{n}^{1,2})=\frac{16r^2-3r+12}{3(7r^2-4r+6))}\approx 0.8064.&\qedhere\end{align}
\end{proof}

\begin{thm}\label{thm:hexa2}$\quad$
	
	$\lim\limits_{n\rightarrow \infty}\avM(G_{n}^{1,1})=\frac{17r^2-13r+8}{3(7r^2-6r+3)}$  where $r=\frac{7+\sqrt[3]{262+6\sqrt{129}}+\sqrt[3]{262-6\sqrt{129}}}{3}.$
	
	$\lim\limits_{n\rightarrow \infty}\avM(G_{n}^{1,3})=\frac{13r^2+36r+28}{3(5r^2+16r+12)}$  where $r=\frac{5+\sqrt[3]{359+12\sqrt{78}}+\sqrt[3]{359-12\sqrt{78}}}{3}.$
\end{thm}
\begin{proof}
	The bivariate generating functions in \cite{Taylor} show that 
	\begin{eqnarray}
	S(G_{n}^{1,1},k)=&4S(G_{n-1}^{1,1},k-2)+3S(G_{n-1}^{1,1},k-3)-4S(G_{n-2}^{1,1},k-4)+3S(G_{n-2}^{1,1},k-5)\nonumber\\
	&-2S(G_{n-2}^{1,1},k-6)+S(G_{n-3}^{1,1},k-6)-2S(G_{n-3}^{1,1},k-7)+2S(G_{n-3}^{1,1},k-8).\nonumber	\end{eqnarray}
	
	We have $I=3$, $\alpha_1=7$, $\alpha_2=-3$, $\alpha_3=1$, $\beta_1=17$, $\beta_2=-13$ and $\beta_3=8$, and the root of maximum modulus of the equation $x^3-7x^2+3x-1=0$ is
	$$r=\frac{7+\sqrt[3]{262+6\sqrt{129}}+\sqrt[3]{262-6\sqrt{129}}}{3}.$$ Since $\maxM(G_{n}^{1,1})=3n$, it follows from Theorem \ref{thm1} that
	$$\lim\limits_{n\rightarrow \infty}\avM(G_{n}^{1,1})=\frac{17r^2-13r+8}{3(7r^2-6r+3))}\approx 0.8234.$$
	Similarly, the bivariate generating functions in \cite{Taylor} show that 
	\begin{eqnarray}
	S(G_{n}^{1,3},k)&=&2S(G_{n-1}^{1,3},k-2)+3S(G_{n-1}^{1,3},k-3)+2S(G_{n-2}^{1,3},k-4)+8S(G_{n-2}^{1,3},k-5)\nonumber\\
	&&-2S(G_{n-2}^{1,3},k-6)+4S(G_{n-3}^{1,3},k-7).\nonumber	\end{eqnarray}
	
	In this case, $I=3$, $\alpha_1=5$, $\alpha_2=8$, $\alpha_3=4$, $\beta_1=13$, $\beta_2=36$ and $\beta_3=28$, and  the root of maximum modulus of the equation $x^3-5x^2-8x-4=0$ is
	$$r=\frac{5+\sqrt[3]{359+12\sqrt{78}}+\sqrt[3]{359-12\sqrt{78}}}{3}.$$ Since $\maxM(G_{n}^{1,3})=3n$, we have
\begin{align}
\lim\limits_{n\rightarrow \infty}\avM(G_{n}^{1,3})=\frac{13r^2+36r+28}{3(5r^2+16r+12))}\approx 0.8257.&\qedhere
\end{align}	
\end{proof}

Instead of chains of hexagons, we could consider chains of cycles of order $r\neq 6$. For example, for $s\in\{1,2\}$, let $G_{n}^{2,s}$ be the graph obtained by considering the disjoint union of $n$ cycles on four vertices, where $\{v_{i,1},\ldots,v_{i,4}\}$ and $\{v_{i,1}v_{i,2},v_{i,2}v_{i,3},v_{i,3}v_{i,4},v_{i,1}v_{i,4}\}$ are the vertex set and the edge set of the $i$th cycle, and by adding an edge between $v_{i,s+1}$ and $v_{i+1,1}$ for $i=1,\ldots,n-1$. Also, let
$G_{n}^{2,s,1}$ be the graph obtained from  $G_{n}^{2,s}$ by adding a vertex $w_1$ linked to $v_{1,1}$, and let
$G_{n}^{2,s,2}$ be the graph obtained from  $G_{n}^{2,s,1}$ by adding a vertex $w_2$ linked to $w_1$.

\begin{thm}$\quad$
	
	$\lim\limits_{n\rightarrow \infty}\avM(G_{n}^{2,1})=\frac{6r^2+r+10}{2(3r^2+2r+6)}$  where $r=1+\sqrt[3]{\frac{45+\sqrt{1257}}{18}}+\sqrt[3]{\frac{45-\sqrt{1257}}{18}}.$
	
	$\lim\limits_{n\rightarrow \infty}\avM(G_{n}^{2,2})=\frac{5r+6}{2(3r+4)}$  where $r=\frac{3+\sqrt{17}}{2}.$
\end{thm}		
\begin{proof}
	Let's start with $G_{n}^{2,1}$. Since exactly one of the edges $v_{1,1}v_{1,4}$ and $v_{1,3}v_{1,4}$ belongs to a maximal matching in $G_{n}^{2,1}$, and since the graph obtained from $G_{n}^{2,1}$ by removing $v_{1,4}$ and one of $v_{1,1},v_{1,3}$ is $G_{n-1}^{2,1,2}$, we have
	\begin{align}
	\Sk{G_{n}^{2,1}}{k}=2\Sk{G_{n-1}^{2,1,2}}{k-1}.\label{eq:cycle1}
	\end{align}
		Let now $M$ be a maximal matching of size $k$ in $G_{n}^{2,1,2}$. If $w_1w_2\in M$, then the other edges of $M$ form a maximal matching in $G_{n}^{2,1}$. If $w_1w_2\notin M$, then $w_1v_{1,1}\in M$ and there are two possible cases: if $v_{1,2}v_{1,3}\in M$,  the other edges of $M$ form a maximal matching in $G_{n-1}^{2,1}$; if $v_{1,2}v_{1,3}\notin M$, then  $v_{1,3}v_{1,4}\in M$ and the other edges of $M$ form a maximal matching in $G_{n-1}^{2,1,1}$. Hence,
	\begin{align}
	\Sk{G_{n}^{2,1,2}}{k}=\Sk{G_{n}^{2,1}}{k-1}+\Sk{G_{n-1}^{2,1}}{k-2}+\Sk{G_{n-1}^{2,1,1}}{k-2}.\label{eq:cycle2}
	\end{align}
	
	Finally, let $M$ be a maximal matching of size $k$ in $G_{n}^{2,1,1}$. If $w_1v_{1,1}\in M$, then as in the previous case, $M$ contains $k-2$ edges in $G_{n-1}^{2,1,1}$ or in $G_{n-1}^{2,1}$. If $w_1v_{1,1}\notin M$, there are two possible cases: if $v_{1,1}v_{1,2}\in M$, then $v_{1,3}v_{1,4}\in M$ and the $k-2$ other edges of $M$ form a maximal matching in $G_{n-1}^{2,1}$; if  $v_{1,1}v_{1,2}\notin M$, then $v_{1,1}v_{1,4}\in M$ and the $k-1$ other edges of $M$ form a maximal matching in $G_{n-1}^{2,1,2}$. Hence, 
	\begin{align}
	\Sk{G_{n}^{2,1,1}}{k}=&2\Sk{G_{n-1}^{2,1}}{k-2}+\Sk{G_{n-1}^{2,1,1}}{k-2}+\Sk{G_{n-1}^{2,1,2}}{k-1}.
	%\nonumber\\
	%=&\Sk{B(n{-}1,1)}{k{-}2}+\Big(\Sk{B(n-1,1)}{k{-}2}+\Sk{B_1(n-1,1)}{k{-}2}\Big)\nonumber\\
	%&+\Sk{B_2(n-1,1)}{k-1}\nonumber\\
	%=&\Sk{B(n{-}1,1)}{k{-}2}{+}\Big(\Sk{B_2(n,1)}{k}{-}\Sk{B(n,1)}{k{-}1}\Big){+}\Sk{B_2(n{-}1,1)}{k{-}1}
	\label{eq:cycle3}
	\end{align}
	Playing with Equations (\ref{eq:cycle1})-(\ref{eq:cycle3}) as we did in Theorem \ref{thm:hexa1} leads to
%	
%	\noindent Combining Equations (\ref{eq:cycle1}) and (\ref{eq:cycle3}) gives
%	\begin{align*}
%	2\Sk{B_1(n,1)}{k}=&2\Sk{B(n{-}1,1)}{k{-}2}{+}\Sk{B(n+1,1)}{k+1}{-}2\Sk{B(n,1)}{k{-}1}{+}\Sk{B(n,1)}{k}
%	\end{align*}
%	which combined with (\ref{eq:cycle2}) gives
%	\begin{align}
%	2\Sk{B_2(n{,}1)}{k}=&2\Sk{B(n,1)}{k-1}+2\Sk{B(n-1,1)}{k-2}+2\Sk{B(n{-}2,1)}{k{-}4}\nonumber\\
%	&{+}\Sk{B(n,1)}{k-1}{-}2\Sk{B(n-1,1)}{k{-}3}{+}\Sk{B(n-1,1)}{k-2}\nonumber\\
%	=&3\Sk{B(n,1)}{k{-}1}+3\Sk{B(n{-}1,1)}{k{-}2}+2\Sk{B(n{-}2,1)}{k{-}4\nonumber}\\
%	&{-}2\Sk{B(n-1,1)}{k{-}3}\label{eq:cycle4}
%	\end{align}
%	Equation (\ref{eq:cycle3}) state that $\Sk{B(n+1,1)}{k+1}=2\Sk{B_2(n,1)}{k}$ which means that
	$$
	\Sk{G_{n}^{2,1}}{k}{=}3\Sk{G_{n-1}^{2,1}}{k{-}2}+3\Sk{G_{n-2}^{2,1}}{k{-}3}{-}2\Sk{G_{n-2}^{2,1}}{k{-}4}
	{+}2\Sk{G_{n-3}^{2,1}}{k{-}5}.$$
	We thus have $I=3$, $\alpha_1=3$, $\alpha_2=1$, $\alpha_3=2$, $\beta_1=6$, $\beta_2=1$ and $\beta_3=10$, and the root of maximum modulus of the equation $x^3-3x^2-x-2=0$ is
	$$r=1+\sqrt[3]{\frac{45+\sqrt{1257}}{18}}+\sqrt[3]{\frac{45-\sqrt{1257}}{18}}.$$ Since $\frac{\maxM(G_{n}^{2,1})}{n}=2$, it follows that
	$$\lim\limits_{n\rightarrow \infty}\avM(G_{n}^{2,1})=\frac{6r^2+r+10}{2(3r^2+2r+6))}\approx 0.8732.$$
	We now consider $G_{n}^{2,2}$. A similar analysis as above leads to the following recurrence relation:
	\begin{align*}
	%\Leftrightarrow\quad&\Sk{B(n+2,2)}{k+2}=\Sk{B(n+1,2)}{k+1}+2\Sk{B(n+1,2)}{k}+2\Sk{B(n,2)}{k-1}\\
	%\Leftrightarrow\quad&
	\Sk{G_{n}^{2,2}}{k}=\Sk{G_{n-1}^{2,2}}{k-1}+2\Sk{G_{n-1}^{2,2}}{k-2}+2\Sk{G_{n-2}^{2,2}}{k-3}.
	\end{align*}
	This time, we have $I=2$, $\alpha_1=3$, $\alpha_2=2$, $\beta_1=5$ and  $\beta_2=6$, and the root of maximum modulus of the equation $x^2-3x-2=0$ is $r=\frac{3+\sqrt{17}}{2}$. Since $\frac{\maxM(G_{n}^{2,2})}{n}=2$, it follows that
\begin{align*}\lim\limits_{n\rightarrow \infty}\avM(G_{n}^{2,2})=\frac{5r+6}{2(3r+4))}=\frac{51+\sqrt{17}}{68}\approx 0.8106.&\qedhere
\end{align*}
\end{proof}

Another chain of cycles is studied in \cite{DK}, where two consecutive cycles share a common vertex. More precicely, let $G_{n}^{3}$ be the graph obtained by considering $n$ disjoint copies of the cycle on $3$ vertices, by choosing two vertices $v_{i,1}$ and $v_{i,2}$ in each  cycle, and by merging (identifying) $v_{i,2}$ and $v_{i+1,1}$ ($i=1,\ldots,n-1$). Also, let $G_{n}^{3-}$ be the graph obtained from $G_{n}^{3}$ by removing $v_{1,1 }$. For illustration, $G_{4}^{3}$ and $G_{4}^{3-}$ are shown in Figure \ref{fig:Hn3}. While the following result is proven in \cite{DK}, we give here a simple proof based on our technique.

\begin{figure}[!ht]
	\centering
	\scalebox{0.9}{\includegraphics{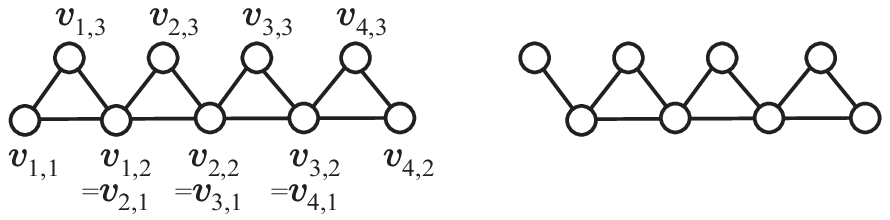}}
	\caption{$G_{4}^{3}$ and $G_{4}^{3-}$.}
	\label{fig:Hn3}
\end{figure}

\begin{thm}$\quad$
	
	$\lim\limits_{n\rightarrow \infty}\avM(G_{n}^{3})=\frac{2r^2+r+2}{2r^2+3}$  where $r=\frac{1}{3}\Big(2+\sqrt[3]{\frac{43+3\sqrt{177}}{2}}+\sqrt[3]{\frac{43-3\sqrt{177}}{2}}\Big).$
\end{thm}		
\begin{proof}
	Let $M$ be a maximal matching of size $k$ in $G_{n}^{3}$. Exactly one of $v_{1,1}v_{1,2}$, $v_{1,1}v_{1,3}$, $v_{1,2}v_{1,3}$ is in $M$. If $v_{1,1}v_{1,3}\in M$, the other edges in $M$ form a maximal matching in $G_{n-1}^{3}$. If $v_{1,1}v_{1,2}$ or $v_{1,2}v_{1,3}$ belongs to $M$, the other edges in $M$ form a maximal matching in $G_{n-1}^{3-}$.
	Hence, 
	\begin{align}
	\Sk{G_{n}^{3}}{k}=\Sk{G_{n-1}^{3}}{k-1}+2\Sk{G_{n-1}^{3-}}{k-1}.\label{eq:hn31}
	\end{align}
	Let now $M$ be a maximal matching of size $k$ in $G_{n}^{3-}$.  If $v_{1,3}v_{2,1}\in M$, the other edges of $M$ form a maximal matching in $G_{n-1}^{3-}$. If $v_{2,1}v_{2,3}\in M$,  the other edges of $M$ form a maximal matching in $G_{n-2}^{3}$. If $v_{2,1}v_{2,2}\in M$, the other edges of $M$ form a maximal matching in $G_{n-2}^{3-}$. Hence,
	\begin{align}
	\Sk{G_{n}^{3-}}{k}=\Sk{G_{n-2}^{3}}{k-1}+\Sk{G_{n-1}^{3-}}{k-1}+\Sk{G_{n-2}^{3-}}{k-1}.\label{eq:hn32}
	\end{align}
	Combining (\ref{eq:hn31}) and (\ref{eq:hn32}) we get
	\begin{align*}
	&\Sk{G_{n{+}1}^{3}}{k{+}1}{-}\Sk{G_{n}^{3}}{k}{=}2\Sk{G_{n{-}2}^{3}}{k{-}1}{+}\Sk{G_{n}^{3}}{k}{-}\Sk{G_{n{-}1}^{3}}{k{-}1}{+}\Sk{G_{n{-}1}^{3}}{k}{-}\Sk{G_{n{-}2}^{3}}{k{-}1}\\
	\Leftrightarrow&\Sk{G_{n+1}^{3}}{k{+}1}=2\Sk{G_{n}^{3}}{k}+\Sk{G_{-1n}^{3}}{k}-\Sk{G_{n-1}^{3}}{k{-}1}+\Sk{G_{n-2}^{3}}{k{-}1}\\
	\Leftrightarrow&\Sk{G_{n}^{3}}{k}=2\Sk{G_{n-1}^{3}}{k{-}1}+\Sk{G_{n-2}^{3}}{k{-}1}-\Sk{G_{n-2}^{3}}{k{-}2}+\Sk{G_{n}^{3}}{k{-}2}.
	\end{align*}
	
	Hence, $I=3$, $\alpha_1=2$, $\alpha_2=0$, $\alpha_3=1$, $\beta_1=2$, $\beta_2=-1$ and  $\beta_3=2$, and the root $r$ of maximum modulus of the equation $x^3-2x^2-1=0$ is
	$$r=\frac{1}{3}\Big(2+\sqrt[3]{\frac{43+3\sqrt{177}}{2}}+\sqrt[3]{\frac{43-3\sqrt{177}}{2}}\Big).$$ Since $\frac{\maxM(G_{n}^{3})}{n}=1$, we have
\begin{align*}
\lim\limits_{n\rightarrow \infty}\avM(G_{n}^{3})=\frac{2r^2-r+2}{2r^2+3}\approx 0.74817.&\qedhere\end{align*}
\end{proof}
Note that $G_n^3$ can be considered as a chain of cliques since $\C{3}\equiv\K{3}$. The next section will consider other chains of cliques.

\subsection{Chains of cliques}\label{sec:chainofcliques}

Let $G_n^{4,s}$ be the graph obtained by taking $n$ disjoint copies of a clique on $s\geq 1$ vertices, by choosing one vertex $v_i$ ($i=1,\ldots,n$) in each clique, by adding $n$ vertices $w_1,\ldots,w_n$ and linking $w_i$ to $v_i$ ($i=1,\ldots,n$), and by linking $v_i$ to $v_{i+1}$ ($i=1,\ldots,n-1$). For illustration, $G_5^{4,1}$ and $G_3^{4,3}$ are depicted in Figure \ref{fig:Hnk}. 
Observe that $G_n^{4,1}=\PathT{n}$. It is proven in Andriantiana \emph{et al.}~\cite{Wagner} that $\lim_{n\rightarrow \infty}\wagner(\Path{n})=\frac{5-\sqrt{5}}{5}$. Therefore, the link between $\wagner(G)$ and $\avM(\GT{G})$ established in Section~\ref{sec_simdiff} implies 
$\avM(\PathT{n})=1-\frac{\wagner(\Path{n})\maxM(\Path{n})}{n}$. Since  $\maxM(\Path{n})=\lfloor\frac{n}{2}\rfloor$, we have
$$\lim\limits_{n\rightarrow \infty}\avM(G_n^{4,1})=1-\lim\limits_{n\rightarrow \infty}\frac{\wagner(\Path{n})\maxM(\Path{n})}{n}=1-\lim\limits_{n\rightarrow \infty}\frac{\frac{5-\sqrt{5}}{5}\lfloor\frac{n}{2}\rfloor}{n}=1-\frac{5-\sqrt{5}}{10}=\frac{5+\sqrt{5}}{10}.$$

\noindent Let $f(n)=\nM(\K{n})$ be the number of maximal matchings in a clique of order $n$. Clearly,
$$f(n)=\begin{cases}
n!!&\mbox{if }n \mbox{ is odd,}\\
(n-1)!!&\mbox{if }n \mbox{ is even,}\\
\end{cases}$$
where $n!!$ is the double factorial that equals $\prod_{k=0}^{\frac{n-1}{2}}(n-2k)$ for an odd number $n$. The following theorem gives the value of $\lim\limits_{n\rightarrow \infty}\avM(G_n^{4,s})$ for all $s\geq 1$.

\begin{figure}[!ht]
	\centering
	\scalebox{0.9}{\includegraphics{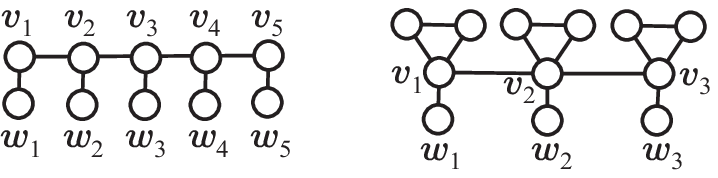}}
	\caption{$G_5^{4,1}$ and $G_3^{4,3}$.}
	\label{fig:Hnk}
\end{figure}

\begin{thm}\label{thm:Hnk}
	$$\displaystyle\lim\limits_{n\rightarrow \infty}\avM(G_n^{4,s})=
	\frac{r\Big(\lfloor\frac{s+1}{2}\rfloor f(s-1)+
		\lceil\frac{s-1}{2}\rceil (s-1) f(s-2)\Big)+
		(2\lfloor\frac{s-1}{2}\rfloor+1)f(s-1)^2}
	{\lceil\frac{s}{2}\rceil\Big(
	r\big(f(s-1)+(s-1)f(s-2)\big)+2f(s-1)^2
	\Big)}$$

where $r=\frac{1}{2}\Big(f(s{-}1){+}(s{-}1)f(s{-}2){+}\sqrt{5f(s{-}1)^2{+}2(s{-}1)f(s{-}1)f(s{-}2){+}(s{-}1)^2f(s{-}2)^2}\Big).$
\end{thm}
\begin{proof}
	Let $M$ be a maximal matching of size $k$ in $G_n^{4,s}$, let $K$ be the clique of order $s$ that contains $v_1$, and let $K'$ be the clique of order $s$ that contains $v_2$. There are three possible cases.
	\begin{itemize}
		\item If $v_1w_1\in M$, then $M$ contains also a maximal matching of $\lfloor\frac{s-1}{2}\rfloor$ edges in $K\setminus \{v_1\}$ and a maximal matching of $k-\lfloor\frac{s+1}{2}\rfloor$ edges in $G_{n-1}^{4,s}$. Since there are $f(s-1)$ matchings of size $\lfloor\frac{s-1}{2}\rfloor$ in $K\setminus \{v_1\}$, the number of matchings of size $k$ in $G_{n}^{4,s}$ that contain $v_1w_1$ is equal to $f(s-1)\Sk{G_{n-1}^{4,s}}{k-\lfloor\frac{s+1}{2}\rfloor}$.
		\item If $s \ge 2$ and $M$ contains one of the $s-1$ edges incident to $v_1$  in $K$, call it $uv_1$, then $M$ also contains a maximal matching of $\lfloor\frac{s-2}{2}\rfloor$ edges in $K$ and a maximal matching of $k-\lceil\frac{s-1}{2}\rceil$ edges in $G_{n-1}^{4,s}$. Since there are $f(s-2)$ matchings of size $\lfloor\frac{s-2}{2}\rfloor$ in $K\setminus \{v_1,u\}$, the number of matchings of size $k$ in $G_{n}^{4,s}$ that contain an edge incident to $v_1$ in $K$ is equal to $(s-1)f(s-2)\Sk{G_{n-1}^{4,s}}{k-\lceil\frac{s-1}{2}\rceil}$.
		\item If $v_1v_2\in M$, then $M$ also contains a maximal matching of $\lfloor\frac{s-1}{2}\rfloor$ edges in $K\setminus \{v_1\}$, a maximal matching of $\lfloor\frac{s-1}{2}\rfloor$  edges in $K'\setminus \{v_2\}$, and a maximal matching of $k-(2\lfloor\frac{s-1}{2}\rfloor+1)$ edges in $G_{n-2}^{4,s}$. Since there are $f(s-1)$ matchings of size $\lfloor\frac{s-1}{2}\rfloor$ in $K\setminus \{v_1\}$ and in $K'\setminus \{v_2\}$, the number of matchings of size $k$ in $G_{n}^{4,s}$ that contain $v_1v_2$ is equal to $f(s-1)^2\Sk{G_{n-2}^{4,s}}{k-2\lfloor\frac{s-1}{2}\rfloor -1}$.
\end{itemize}
We have thus shown that
\begin{align*}
S(G_{n}^{4,s})=&f(s-1)\Sk{G_{n-1}^{4,s}}{k-\left\lfloor\frac{s+1}{2}\right\rfloor}+
(s-1)f(s-2)\Sk{G_{n-1}^{4,s}}{k-\left\lceil\frac{s-1}{2}\right\rceil}\\
&+f(s-1)^2\Sk{G_{n-2}^{4,s}}{k-2\left\lfloor\frac{s-1}{2}\right\rfloor -1},
\end{align*}
which means that $I=2$, %$a_{1,\lfloor\frac{s+1}{2}\rfloor}=f(s-1)$, $a_{1,\lceil\frac{s-1}{2}\rceil}=(s-1)f(s-2)$ and $a_{2,2\lfloor\frac{s-1}{2}\rfloor -1}=f(s-1)^2$ which implies 
$\alpha_1=f(s-1)+(s-1)f(s-2)$, $\alpha_2=f(s-1)^2$, 
$\beta_1=\lfloor\frac{s+1}{2}\rfloor f(s-1)+\lceil\frac{s-1}{2}\rceil (s-1)f(s-2)$ and $\beta_2=(2\lfloor\frac{s-1}{2}\rfloor +1)f(s-1)^2$.
The root of maximum modulus of the equation $x^2-\alpha_1x-\alpha_2=0$ is $r=\frac{1}{2}(\alpha_1 +\sqrt{\alpha_1^2+4\alpha_2})$ which gives
\begin{align*}
r=&\frac{1}{2}(f(s-1)+(s-1)f(s-2)+\sqrt{(f(s-1)+(s-1)f(s-2))^2+4f(s-1)^2}\\
	=&\frac{1}{2}(f(s-1)+(s-1)f(s-2)+\sqrt{5f(s-1)^2+2(s-1)f(s-1)f(s-2)+(s-1)^2f(s-2)^2}.
\end{align*}
To conclude the proof, it is sufficient to observe that  $\maxM(G_{n}^{4,s})=\lceil\frac{s}{2}\rceil n$, which implies $\lim\limits_{n\rightarrow \infty}\frac{\maxM(G_{n}^{4,s})}{n}=\lceil\frac{s}{2}\rceil$.
	\end{proof}
	
\noindent For illustration, for $s=1$, Theorem \ref{thm:Hnk} gives $\lim\limits_{n\rightarrow \infty}\avM(G_{n}^{4,1})=\frac{r+1}{r+2}$ with $r=\frac{1+\sqrt{5}}{2}$, which, as expected, implies
$$\lim\limits_{n\rightarrow \infty}\avM(G_{n}^{4,1})=\frac{3+\sqrt{5}}{5+\sqrt{5}}=\frac{5+\sqrt{5}}{10}\approx 0.7236.$$
The smallest value of $\lim\limits_{n\rightarrow \infty}\avM(G_{n}^{4,s})$ is obtained with $s=3$, where
$\lim\limits_{n\rightarrow \infty}\avM(G_{n}^{4,3})=\frac{4r+3}{6r+4}$ with $r=\frac{3+\sqrt{13}}{2}$, which implies
$$\lim\limits_{n\rightarrow \infty}\avM(G_{n}^{4,3})=\frac{18+4\sqrt{13}}{26+6\sqrt{13}}=\frac{39-\sqrt{13}}{52}\approx 0.6807.$$
The values of $\lim\limits_{n\rightarrow \infty}\avM(G_{n}^{4,s})$ for $s\leq 50$ are shown in Figure \ref{fig:InvHnk}. 

\begin{figure}[!ht]
	\centering
	\scalebox{0.55}{\includegraphics{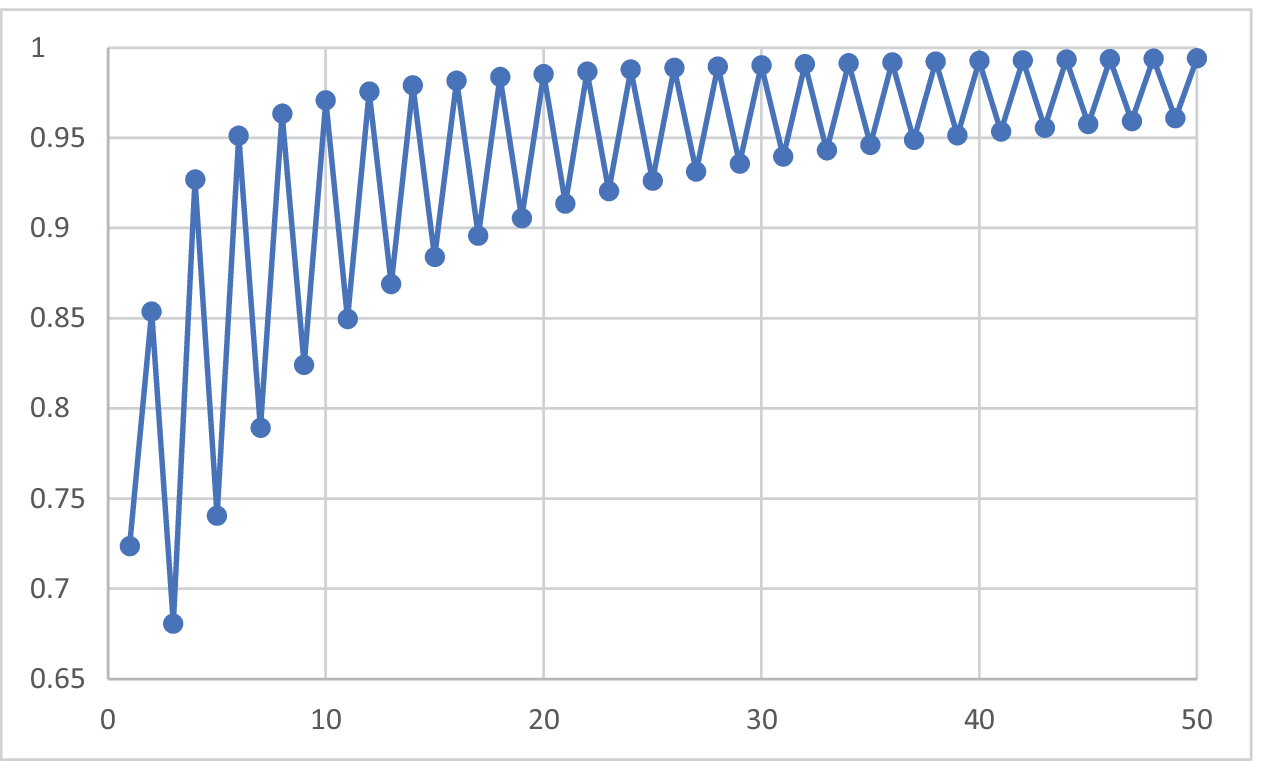}}
	\caption{$\lim\limits_{n\rightarrow \infty}\avM(G_{n}^{4,s})$ for $s\leq 50$.}
	\label{fig:InvHnk}
\end{figure}
\begin{cor}
	$\lim\limits_{n\rightarrow \infty}\avM(\CT{n})=\lim\limits_{n\rightarrow \infty}\avM(\PathT{n}).$
\end{cor}
\begin{proof}
	Let $\C{n}$ be a cycle on $n$ vertices $v_1,\ldots,v_n$ with edges $v_iv_{i+1}$ ($i=1,\ldots,n{-}1$) and $v_1v_n$, and let $w_i$ be the vertex of degree 1 linked to $v_i$ ($i=1,\ldots,n$) in $\CT{n}$. Also, let $M$ be a maximal matching of size $k$ in $\CT{n}$. If $w_1v_1\in M$, then $M\setminus\{w_1v_1\}$ is a maximal matching in  $\PathT{n-1}$ obtained from $\CT{n}$ by removing $v_1$ and $w_1$. If $w_1v_1\notin M$, then $M$ contains $v_1v_2$ or $v_1v_n$ as well as a maximal matching in $\PathT{n-2}$ obtained from $\CT{n}$ by removing $w_1, v_1$ and one of $v_2,v_n$. Hence
	$$\Sk{\CT{n}}{k}=\Sk{\PathT{n-1}}{k-1}+2\Sk{\PathT{n-2}}{k-1}.$$
	Theorem \ref{thm:Hnk} implies
	$$\Sk{\PathT{n}}{k}=\Sk{\PathT{n-1}}{k-1}+\Sk{\PathT{n-2}}{k-1}$$
	which gives,
\begin{align*}\Sk{\CT{n}}{k}=&\Sk{\PathT{n-2}}{k-2}+\Sk{\PathT{n-3}}{k-2}+2\Sk{\PathT{n-3}}{k-2}+2\Sk{\PathT{n-4}}{k-2}\\
=&\Sk{\CT{n-1}}{k-1}+\Sk{\CT{n-2}}{k-1}.
\end{align*}
This is the same recurrence relation as in Theorem \ref{thm:Hnk} for $s=1$. Since $\maxM(\CT{n})=\maxM(\PathT{n})=n$, we conclude that 
$\lim\limits_{n\rightarrow \infty}\avM(\CT{n})=\lim\limits_{n\rightarrow \infty}\avM(\PathT{n}).$
%With the notations of Theorem \ref{thm1}, we have $I=2, \alpha_1=\alpha_2=\beta_1=\beta_2=1$. Since the root of maximum modulus of the equation $x^2-x-1=0$ has value $\frac{1+\sqrt{5}}{2}$, we get
%$$\lim\limits_{n\rightarrow \infty}\avM(\C{n}^+)=\frac{\frac{1+\sqrt{5}}{2}+1}{\frac{1+\sqrt{5}}{2}+2}=\frac{3+\sqrt{5}}{5+\sqrt{5}}=\frac{5+\sqrt{5}}{10}.$$
\end{proof}

We now consider another kind of chain of cliques. For $s\geq 2$ let $G_{n}^{5,s}$ be the graph obtained 
by taking $n$ disjoint copies of a clique on $s$ vertices, by choosing two vertices $v_{i,1}$ and $v_{i,2}$ in each clique, and by linking $v_{i,2}$ to $v_{i+1,1}$ ($i=1,\ldots,n-1$). 
Also, let $G_{n}^{5,s+}$ be the graph obtained from $G_{n}^{5,s}$ by adding a vertex $w$ linked to $v_{1,1}$ and, if $s\geq 3$, let $G_{n}^{5,s-}$ be the graph obtained from $G_{n}^{5,s}$ by removing a vertex $w\neq v_{1,1},v_{1,2}$ in the clique of $s$ vertices that contains $v_{1,1}$ and $v_{1,2}$,  For illustration, $G_{3}^{5,4}$, $G_{3}^{5,4+}$ and $G_{3}^{5,4-}$ are depicted in Figure \ref{fig:Hprimenk}.
\begin{figure}[!ht]
	\centering
	\scalebox{0.9}{\includegraphics{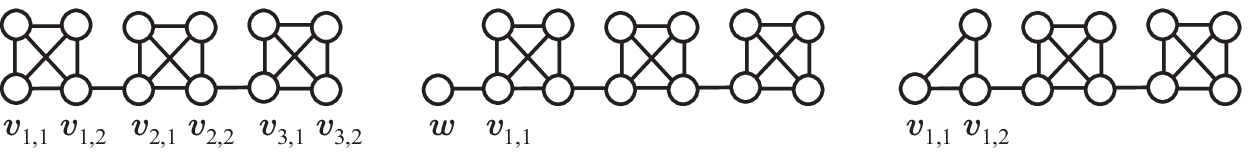}}
	\caption{$G_{3}^{5,4}$, $G_{3}^{5,4+}$ and $G_{3}^{5,4-}$.}
	\label{fig:Hprimenk}
\end{figure}	
	\begin{thm}\label{thm:H'nk}$ $
	\noindent If $s \ge 3$ is odd, then
	$$\displaystyle\lim\limits_{n\rightarrow \infty}\avM(G_{n}^{5,s})=
	\frac{r^2\frac{s^2-1}{2}f(s-2)-r(s^2-s-1)f(s-2)^2+\frac{2(s-1)}{2}f(s-2)^3}{\frac{s}{2}(r^2(s+1)f(s-2)-2rsf(s-2)^2+3f(s-2)^3)}$$
		where $r$ is the root of maximum modulus of the equation 
	$$x^3-(s+1)f(s-2)x^2+sf(s-2)^2x-f(s-2)^3=0,$$
		\noindent while if $s \ge 2$ is even, then
	$$\displaystyle\lim\limits_{n\rightarrow \infty}\avM(G_{n}^{5,s})=
	\frac{r^2\frac{s^2}{2}f(s-2)+rs(s-1)(s-3)f(s-2)^2+\frac{3s^2-5s+2}{2}f(s-2)^3}{\frac{s}{2}(r^2sf(s-2)+2r(s^2-3s+1)f(s-2)^2+3(s-1)f(s-2)^3)}$$
		where $r$ is the root of maximum modulus of the equation 
	$$x^3-sf(s-2)x^2-(s^2-3s+1)f(s-2)^2x-(s-1)f(s-2)^3=0.$$
	\end{thm}
	 
	\begin{proof}
		Let's first assume that $s$ is odd, let $M$ be a maximal matching of size $k$ in $G_{n}^{5,s}$, and let $K$ be the clique of order $s$ that contains $v_{1,1}$ and $v_{1,2}$. If $M$ contains one of the $s-1$ edges incident to $v_{1,2}$ in $K$, call it $uv_{1,2}$, then $M$ also contains a maximal matching  of $\frac{s-3}{2}$ edges in $K\setminus \{u,v_{1,2}\}$  and a maximal matching of $k - \frac{s-1}{2}$ edges in $G_{n-1}^{5,s}$. On the other hand, if $M$ has no edge of $K$ incident to $v_{1,2}$, then it contains a maximal matching of $\frac{s-1}{2}$ edges of $K\setminus \{v_{1,2}\}$ as well as a maximal matching of $k - \frac{s-1}{2}$ edges in $G_{n-1}^{5,s+}$. Since there are $f(s-2)$ maximal matchings in cliques of order $s-1$ and $s-2$, we get:
		\begin{align}\Sk{G_{n}^{5,s}}{k}=(s-1)f(s-2)\Sk{G_{n-1}^{5,s}}{k-{\frac{s-1}{2}}}+f(s-2)\Sk{G_{n-1}^{5,s+}}{k-{\frac{s-1}{2}}}.\label{eq4}\end{align}
		Consider now a maximal matching $M$ of size $k$ in $G_{n}^{5,s-}$. If $M$ contains one of the $s-2$ edges incident to $v_{1,2}$ in $K$, call it $uv_{1,2}$, then $M$ also contains   a maximal matching of $\frac{s-3}{2}$ edges in $K\setminus \{w,u,v_{1,2}\}$ and a maximal matching of $k - \frac{s-1}{2}$ edges in $G_{n-1}^{5,s}$. Otherwise, $v_{1,2}v_{2,1}\in M$ and $M$ contains a maximal matching of $\frac{s-3}{2}$  edges in $K\setminus \{w,v_{1,2}\}$ and a maximal matching of  $k - \frac{s-1}{2}$ edges in $G_{n-1}^{5,s-}$. Since there are $f(s-3)$ maximal matchings in cliques of order $s-3$ and $f(s-2)$ maximal matchings in cliques of order $s-2$, we get:
	\begin{align}\Sk{G_{n}^{5,s-}}{k}=(s-2)f(s-3)\Sk{G_{n-1}^{5,s}}{k-{\frac{s-1}{2}}}+f(s-2)\Sk{G_{n-1}^{5,s-}}{k-{\frac{s-1}{2}}}.\label{eq5}\end{align}
	Consider finally a maximal matching $M$ of size $k$ in $G_{n}^{5,s+}$. If $wv_{1,1}\in M$, then $M$ contains a maximal matching of size $k-1$ in $G_{n}^{5,s-}$. If $v_{1,1}v_{1,2}\in M$, then $M$ contains a maximal matching of $\frac{s-3}{2}$ edges in $K\setminus\{v_{1,1},v_{1,2}\}$ and a maximal matching of $k - \frac{s-1}{2}$ edges in $G_{n-1}^{5,s}$. The remaining case is when $M$ contains one of the $s-2$ edges incident to $v_{1,1}$ in $K\setminus \{v_{1,1},v_{1,2}\}$, call it $u,v_{1,1}$. There are  two possibilities: if $M$ contains an edge $u'v_{1,2}$ of $K\setminus \{u,v_{1,1}\}$, then $s \ge 5$ and $M$ also contains a maximal matching of $\frac{s-5}{2}$ edges in $K\setminus \{u,u',v_{1,1},v_{1,2}\}$ and a maximal matching of $k - \frac{s-1}{2}$ edges in $G_{n-1}^{5,s}$; if no edge in $M$ is incident to $v_{1,2}$ in $K$, then $M$ contains a maximal matching of $\frac{s-3}{2}$ edges in $K\setminus \{u,v_{1,1},v_{1,2}\}$ and a maximal matching of $k - \frac{s-1}{2}$ edges in $G_{n-1}^{5,s+}$. Altogether this gives
\begin{align}
\Sk{G_{n}^{5,s+}}{k}=&\Sk{G_{n}^{5,s-}}{k-1}+f(s-2)\Sk{G_{n-1}^{5,s}}{k{-}{\frac{s-1}{2}}}\nonumber\\&+(s{-}2)\Big((s{-}3)f(s{-}3)\Sk{G_{n-1}^{5,s}}{k{-}{\frac{s{-}1}{2}}}+f(s{-}3)\Sk{G_{n-1}^{5,s+}}{k{-}{\frac{s{-}1}{2}}}\Big).\label{eq6}
\end{align}
Using the same proof technique as in Theorem \ref{thm:hexa1}, it is not difficult to show that playing with Equations (\ref{eq4})-(\ref{eq6}) leads to
		\begin{align}\Sk{G_{n}^{5,s}}{k}=&(s+1)f(s-2)\Sk{G_{n-1}^{5,s}}{k-{\frac{s-1}{2}}}-(s+1)f(s-2)^2\Sk{G_{n-2}^{5,s}}{k-s+1}\nonumber\\
		&+f(s-2)^2\Sk{G_{n-2}^{5,s}}{k-s}
		+f(s-2)^3\Sk{G_{n-3}^{5,s}}{k-\frac{3(s-1)}{2}}.\label{eq7}\end{align}
		
Hence, we have  $I{=}3$, 
$\alpha_1{=}(s{+}1)f(s{-}2)$,
$\alpha_2{=}-sf(s{-}2)^2$,
$\alpha_3{=}f(s{-}2)^3$,
$\beta_1{=}\frac{s^2{-}1}{2}f(s{-}2)$,
$\beta_2{=}(-s^2{+}s{+}1)f(s{-}2)^2$, and
$\beta_3=\frac{3(s{-}1)}{2}f(s{-}2)^3$.
Since $\maxM(G_{n}^{5,s})=\frac{ns-1}{2}$, we have $\lim_{n\rightarrow \infty}\frac{\maxM(G_{n}^{5,s})}{n}=\frac{s}{2}$, which means that by setting $r$ equal to the root of maximum modulus of the equation
$x^3-\alpha_1x^2-\alpha_2x-\alpha_3=0$, one gets the result in the statement of the theorem.

\noindent For $s$ even, we have the following relations obtained with a similar analysis as for $s$ odd:
\begin{align}
\Sk{G_{n}^{5,s}}{k}=&f(s-1)\Sk{G_{n-1}^{5,s}}{k-{\frac{s}{2}}}+(s-2)f(s-1)f(s-2)\Sk{G_{n-1}^{5,s}}{k-s+1}\nonumber\\
&+f(s-1)f(s-2)\Sk{G_{n-2}^{5,s+}}{k-s+1}.\label{eq8}\\
\Sk{G_{n}^{5,s+}}{k}=&(f(s-1)+(s-2)f(s-2))\Sk{G_{n-1}^{5,s}}{k-{\frac{s}{2}}}
+(s-2)^2f(s-2)^2\Sk{G_{n-2}^{5,s}}{k-s+1}\nonumber\\
&+f(s-2)\Sk{G_{n-1}^{5,s+}}{k-\frac{s}{2}}+(s-2)f(s-2)^2\Sk{G_{n-2}^{5,s+}}{k-s+1}.
\end{align}
By playing with these two equations as we did for $s$ odd, we get
\begin{align}
\Sk{G_{n}^{5,s}}{k}=&sf(s-2)\Sk{G_{n-1}^{5,s}}{k-{\frac{s}{2}}}
+s(s-2)f(s-2)^2\Sk{G_{n-2}^{5,s}}{k-s+1}\nonumber\\
&-(s-1)f(s-2)^2\Sk{G_{n-2}^{5,s}}{k-s}
+(s-1)f(s-2)^3\Sk{G_{n-3}^{5,s}}{k-\frac{3s-2}{2}}.\label{eq:chainclique}
\end{align}
Hence, we have $I=3$, 
	$\alpha_1=sf(s-2)$,
	$\alpha_2=(s^2-3s+1)f(s-2)^2$,
	$\alpha_3=(s-1)f(s-2)^3$,
	$\beta_1=\frac{s^2}{2}f(s-2)$,
	$\beta_2=s(s-1)(s-3)f(s-2)^2$,
	$\beta_3=\frac{3s^2-5s+2}{2}f(s-2)^3$.
Since $\frac{\maxM(G_{n}^{5,s})}{n}=\frac{s}{2}$, by setting $r$ equal to the root of maximum modulus of the equation
$x^3-\alpha_1x^2-\alpha_2x-\alpha_3=0$, one gets the result in the statement of the theorem.
\end{proof}

\begin{figure}[!ht]
	\centering
	\scalebox{0.6}{\includegraphics{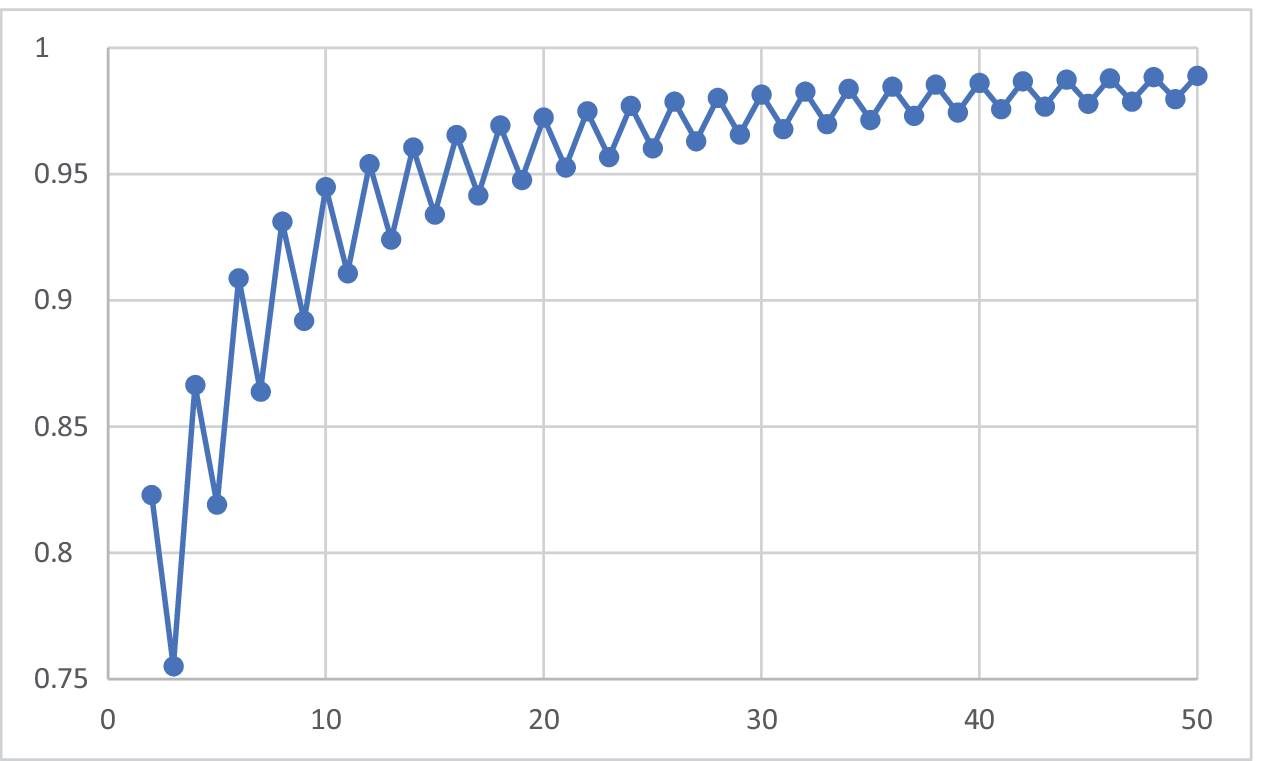}}
	\caption{$\lim\limits_{n\rightarrow \infty}\avM(G_{n}^{5,s})$ for $s\leq 50$.}
	\label{fig:InvHPrimenk}
\end{figure}
The values of $\lim\limits_{n\rightarrow \infty}\avM(G_{n}^{5,s})$ for $s\leq 50$ are shown in Figure \ref{fig:InvHPrimenk}.
The lowest value is obtained with $s=3$ for which we get the following recurrence equation 
$$\Sk{G_{n}^{5,3}}{k}=4\Sk{G_{n-1}^{5,3}}{k-1}
-4\Sk{G_{n-2}^{5,3}}{k-2}+\Sk{G_{n-2}^{5,3}}{k-3}
+\Sk{G_{n-3}^{5,3}}{k-3}.
$$
The root with maximum modulus of the equation $x^3-4x^2+3x-1=0$ is
$$r=\frac{1}{3}\Big(4+\sqrt[3]{\frac{47+3\sqrt{93}}{2}}+\sqrt[3]{\frac{47-3\sqrt{93}}{2}}\Big)$$
which gives
$$\lim\limits_{n\rightarrow \infty}\avM(G_{n}^{5,3})=\frac{2(4r^2-5r+3)}{3(4r^2-6r+3)}\approx 0.75503.$$
\noindent Note that $G_{n}^{5,2}\equiv\Path{2n}$ and, for $s=2$, Equation (\ref{eq:chainclique}) gives
$$\Sk{G_{n}^{5,2}}{k}=2\Sk{G_{n-1}^{5,2}}{k-1}
-\Sk{G_{n-2}^{5,2}}{k-2}
+\Sk{G_{n-3}^{5,2}}{k-2}
$$
which we could have derived from Theorem \ref{thm:path}. Indeed:
\begin{align*}
\Sk{G_{n}^{5,2}}{k}=&\Sk{P_{2n}}{k}=\Sk{P_{2n-2}}{k-1}+\Sk{P_{2n-3}}{k-1}\\
=&\Sk{P_{2n-2}}{k-1}+\Sk{P_{2n-5}}{k-2}+\Sk{P_{2n-6}}{k-2}\\
=&\Sk{P_{2n-2}}{k-1}+(\Sk{P_{2n-2}}{k-1}-\Sk{P_{2n-4}}{k-2})+\Sk{P_{2n-6}}{k-2}\\
=&2\Sk{G_{n-1}^{5,2}}{k-1}-\Sk{G_{n-2}^{5,2}}{k-2}+\Sk{G_{n-3}^{5,2}}{k-2}.
\end{align*}

\subsection{Ladders}
Let $G_n^6$ be the Cartesian product of $P_2$ with $P_n$ which is a \emph{ladder} with vertices $v_i,w_i$ ($i{=}1,{\ldots},n$) and edges $v_iv_{i+1} ,w_iw_{i+1}(i{=}1,{\ldots},n{-}1)$ and $v_iw_i$ $(i{=}1,{\ldots},n)$. Also, let $G_n^{6,1}$ be the graph obtained from $G_n^6$ by adding a vertex $u_1$ linked to $v_1$ and let $G_n^{6,2}$ be the graph obtained from $G_n^{6,1}$ by adding a vertex $u_2$ linked to $u_1$. For illustration, {$G_{3}^{6}$, $G_{3}^{6,1}$ and $G_{3}^{6,2}$ are  depicted in Figure \ref{fig:grid}.

\begin{thm}$\quad$	
	$\lim\limits_{n\rightarrow \infty}\avM(G_n^6)=\frac{2r^4-r^2+3r+4}{2r^4+4r+5)}$  where $r=\frac{1}{3}\Big(1+\sqrt[3]{\frac{47+3\sqrt{93}}{2}}+\sqrt[3]{\frac{47-3\sqrt{93}}{2}}\Big).$
\end{thm}		
\begin{proof}
		Let $M$ be a maximal matching of size $k$ in $G_n^{6}$.
	If $v_1w_1\in M$, then $M\setminus\{v_1w_1\}$ is a maximal matching in $G_{n-1}^{6}$. If $v_1v_2\in M$,  then $M\setminus\{v_1v_2\}$ is a maximal matching in $G_{n-1}^{6,2}$. If $w_1w_2\in M$ , then $M\setminus\{w_1w_2\}$ is a maximal matching in  $G_{n-2}^{6,2}$. If $M$ contains both $v_1v_2$ and $w_1w_2$, then $M\setminus\{v_1v_2,w_1w_2\}$ is a maximal matching in  $G_{n-2}^{6}$. In summary:
\begin{align}
\Sk{G_{n}^{6}}{k}=\Sk{G_{n-1}^{6}}{k-1}-\Sk{G_{n-2}^{6}}{k-2}+2\Sk{G_{n-2}^{6,2}}{k-1}.\label{eq:grid1}
\end{align}
	Let now $M$ be a maximal matching of size $k$ in $G_{n}^{6,1}$. If $u_1v_1\in M$, then $M\setminus\{u_1v_1\}$ is a maximal matching in $G_{n-1}^{6,1}$. If $v_1v_2\in M$, then $M\setminus\{v_1v_2\}$ is  a maximal matching in $G_{n-2}^{6,2}$. If $v_1w_1\in M$,  then $M\setminus\{v_1w_1\}$ is a maximal matching in $G_{n-1}^{6}$. Hence  
\begin{align}
\Sk{G_{n}^{6,1}}{k}=\Sk{G_{n-1}^{6}}{k-1}+\Sk{G_{n-1}^{6,1}}{k-1}+\Sk{G_{n-2}^{6,2}}{k-1}.\label{eq:grid2}
\end{align}
	Finally, let $M$ be a maximal matching of size $k$ in $G_{n}^{6,2}$. If $u_1u_2\in M$, then $M\setminus\{u_1u_2\}$ is  a maximal matching in $G_{n}^{6}$. If $u_1v_1\in M$, then $M\setminus\{u_1v_1\}$ is a maximal matching in $G_{n-1}^{6,1}$. Hence  
		\begin{align}
	\Sk{G_{n}^{6,2}}{k}=\Sk{G_{n}^{6}}{k-1}+\Sk{G_{n-1}^{6,1}}{k-1}.\label{eq:grid3}
	\end{align}
Equations (\ref{eq:grid1}) and (\ref{eq:grid3}) give
\begin{align}\Sk{G_{n}^{6}}{k}=&\Sk{G_{n-1}^{6}}{k-1}+\Sk{G_{n-2}^{6}}{k-2}+2\Sk{G_{n-3}^{6,1}}{k-2}\label{eq:grid4}\\
2\Sk{G_{n}^{6,1}}{k}=&2\Sk{G_{n{-}1}^{6}}{k{-}1}+2\Sk{G_{n{-}1}^{6,1}}{k{-}1}{+}\Big(\Sk{G_{n}^{6}}{k}{-}\Sk{G_{n{-}1}^{6}}{k{-}1}{+}\Sk{G_{n{-}2}^{6}}{k{-}2}\Big)\nonumber\\
=&\Sk{G_{n}^{6}}{k}+\Sk{G_{n-1}^{6}}{k-1}+\Sk{G_{n-2}^{6}}{k-2}+2\Sk{G_{n-1}^{6,1}}{k-1}.\label{eq:grid5}
\end{align}
By combining Equations (\ref{eq:grid4}) and (\ref{eq:grid5}) we get
\begin{align*}\Sk{G_{n}^{6}}{k}=&\Big(2\Sk{G_{n-1}^{6}}{k-1}-\big(\Sk{G_{n-2}^{6}}{k-2}+\Sk{G_{n-3}^{6}}{k-3}+2\Sk{G_{n-4}^{6,1}}{k-3}\big)\Big)\\
&{+}\Sk{G_{n{-}2}^{6}}{k{-}2}{+}\Big(\Sk{G_{n{-}3}^{6}}{k{-}2}{+}\Sk{G_{n{-}4}^{6}}{k{-}3}{+}\Sk{G_{n{-}5}^{6}}{k{-}4}{+}2\Sk{G_{n{-}4}^{6,1}}{k{-}3}\Big)\\
=&2\Sk{G_{n-1}^{6}}{k{-}1}+\Sk{G_{n-3}^{6}}{k{-}2}-\Sk{G_{n-3}^{6}}{k{-}3}+\Sk{G_{n-4}^{6}}{k{-}3}+\Sk{G_{n-5}^{6}}{k{-}4}.
\end{align*}
We have $I{=}5$, $\alpha_1{=}2$, $\alpha_2{=}\alpha_3{=}0$, $\alpha_4{=}\alpha_5{=}1$, $\beta_1{=}2$, $\beta_2{=}0$, $\beta_3{=}-1$, $\beta_4{=}3$ and  $\beta_5=4$. The root of maximum modulus of the equation $x^5-2x^4-x-1=0$ is
$$r=\frac{1}{3}\Big(1+\sqrt[3]{\frac{47+3\sqrt{93}}{2}}+\sqrt[3]{\frac{47-3\sqrt{93}}{2}}\Big).$$ Since $\frac{\maxM(G_{n}^{6})}{n}=1$, we have
\begin{align*}
\lim\limits_{n\rightarrow \infty}\avM(G_{n}^{6})=\frac{2r^4-r^2+3r+4}{2r^4+4r+5}\approx 0.8618.&\qedhere\end{align*}
\end{proof}

We now consider $\widetilde{G}_{n}^{6}$ as well as $\widetilde{G}_{n}^{6,+}$ which is obtained from $\widetilde{G}_{n}^{6}$ by adding a vertex $u_1$ linked to $v_1$ and a vertex $u_2$ linked to $u_1$. For illustration, $\widetilde{G}_{3}^{6}$ and $\widetilde{G}_{3}^{6+}$ are depicted in Figure \ref{fig:grid}.

\begin{figure}[!ht]
	\centering
	\scalebox{0.85}{\includegraphics{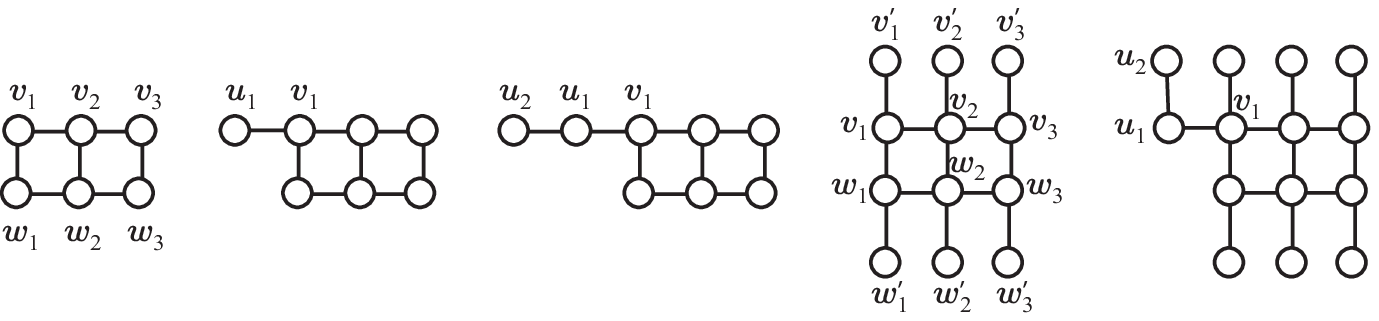}}
	\caption{$G_{3}^{6}$, $G_{3}^{6,1}$, $G_{3}^{6,2}$, $\widetilde{G}_{3}^{6}$ and $\widetilde{G}_{3}^{6+}$.}
	\label{fig:grid}
\end{figure}

\begin{thm}$\quad$	
	$\lim\limits_{n\rightarrow \infty}\avM(\widetilde{G}_{n}^{6})=\frac{4r^2+3r-3}{2(3r^2+2r-3)}$  where $r=1+\frac{12+\sqrt[3]{54(-5+3i\sqrt{111})}}{3\sqrt[3]{27+3i\sqrt{111}}}.$
\end{thm}		
\begin{proof}
Let $M$ be a maximal matching of size $k$ in $\widetilde{G}_{n}^{6}$. If $v_1v_1'\in M$, then $M\setminus\{v_1v_1'\}$ is a maximal matching  in $\widetilde{G}_{n-1}^{6+}$. If $v_1w_1\in M$, then $M\setminus\{v_1w_1\}$ is a maximal matching  in $\widetilde{G}_{n-1}^{6}$. If $v_1v_2\in M$, then there are two possibilies: if $w_1,w_1'\in M$, then $M\setminus\{v_1v_2,w_1w_1'\}$ is a maximal matching  in $\widetilde{G}_{n-2}^{6+}$; if $w_1,w_2\in M$, then $M\setminus\{v_1v_2,w_1w_2\}$ is a maximal matching  in  $\widetilde{G}_{n-2}^{6}$. In summary 
\begin{align}
\Sk{\widetilde{G}_{n}^{6}}{k}=\Sk{\widetilde{G}_{n-1}^{6}}{k-1}+\Sk{\widetilde{G}_{n-2}^{6}}{k-2}+\Sk{\widetilde{G}_{n-1}^{6+}}{k-1}+\Sk{\widetilde{G}_{n-2}^{6+}}{k-2}.\label{eq:grid+1}
\end{align}

Let $M$ be a maximal matching of size $k$ in $\widetilde{G}_{n}^{6+}$. If $u_1u_2\in M$, then $M\setminus\{u_1u_2\}$ is a maximal matching  in $\widetilde{G}_{n}^{6}$. If $u_1v_1\in M$, then $M\setminus\{u_1v_1\}$ is a maximal matching  in $\widetilde{G}_{n-1}^{6+}$. Hence
\begin{align}
\Sk{\widetilde{G}_{n}^{6+}}{k}=\Sk{\widetilde{G}_{n}^{6}}{k-1}+\Sk{\widetilde{G}_{n-1}^{6+}}{k-1}.\label{eq:grid+2}
\end{align}
Equations (\ref{eq:grid+1}) and (\ref{eq:grid+2}) imply
\begin{align*}
\Sk{\widetilde{G}_{n}^{6}}{k}=&\Sk{\widetilde{G}_{n-1}^{6}}{k-1}+\Sk{\widetilde{G}_{n-1}^{6}}{k-2}+\Sk{\widetilde{G}_{n-2}^{6}}{k-2}+2\Sk{\widetilde{G}_{n-2}^{6+}}{k-2}\\
=&2\Sk{\widetilde{G}_{n-1}^{6}}{k{-}1}{-}\Big(\Sk{\widetilde{G}_{n-2}^{6}}{k{-}2}{+}\Sk{\widetilde{G}_{n-2}^{6}}{k{-}3}{+}\Sk{\widetilde{G}_{n-3}^{6}}{k{-}3}{+}2\Sk{\widetilde{G}_{n-3}^{6+}}{k{-}3}\Big)\\
&+\Sk{\widetilde{G}_{n-1}^{6}}{k-2}+\Sk{\widetilde{G}_{n-2}^{6}}{k-2}+2\Big(\Sk{\widetilde{G}_{n-2}^{6}}{k-3}+\Sk{\widetilde{G}_{n-3}^{6+}}{k-3}\Big)\\
=&2\Sk{\widetilde{G}_{n-1}^{6}}{k{-}1}+\Sk{\widetilde{G}_{n-1}^{6}}{k-2}{+}\Sk{\widetilde{G}_{n-2}^{6}}{k{-}3}{-}\Sk{\widetilde{G}_{n-3}^{6}}{k{-}3}.
\end{align*}
We have $I=3$, $\alpha_1=3$, $\alpha_2=1$, $\alpha_3=-1$, $\beta_1=4$, $\beta_2=3$ and  $\beta_3=-3$, and the root of maximum modulus of the equation $x^3-3x^2-x+1=0$ is
$$r=1+\frac{12+\sqrt[3]{54(-5+3i\sqrt{111})}}{3\sqrt[3]{27+3i\sqrt{111}}}.$$ Since $\frac{\maxM(\widetilde{G}_{n}^{6})}{n}=2$, we have
\begin{align*}\lim\limits_{n\rightarrow \infty}\avM(\widetilde{G}_{n}^{6})=\frac{4r^2+3r-3}{2(3r^2+2r-3)}\approx 0.6968.&\qedhere\end{align*}
	\end{proof}

\subsection{Trees}\label{sec:trees}
Let's now talk about trees. We have mentioned that Andriantiana {\it et al.}\cite{Wagner} have proven that the trees of order $n$ that maximize the average size of a (not necessarily maximal) matching are the paths $\Path{n}$. So let $\Tset_n$ be the set of trees of order $n$ and let $\widetilde\Tset_n$ be the set of trees ${\widetilde{T}}$ with $T\in \Tset_n$. The result in \cite{Wagner} can ve rewritten as 
$$\max_{T\in \Tset_n}\frac{\totMw(T)}{\nMw(T)}=\frac{\totMw(\Path{n})}{\nMw(\Path{n})}\quad\Leftrightarrow\quad \max_{T\in \Tset_n}\maxM(T)\wagner(T)=\maxM(\Path{n})\wagner(\Path{n}).$$
Since $\maxM(\widetilde{T})=n$ for all $\widetilde{T}\in{\widetilde\Tset}_n$, we have
$$\min_{\widetilde{T}\in\widetilde\Tset_n}\avM(\widetilde{T})=1-\max_{T\in\Tset_n}\frac{\maxM(T)\wagner(T)}{n}=1-\frac{\maxM(\Path{n})\wagner(\Path{n})}{n}=\avM(\PathT{n}).$$
Clearly, $\widetilde\Tset_n\subset \Tset_{2n}$ and Dyer and Frieze \cite{DF} conjecture that 
$$\lim\limits_{n\rightarrow \infty} \min_{T\in \Tset_{2n}}\dyer(T)=
\lim\limits_{n\rightarrow \infty}\dyer(\PathT{n}).$$
We show here that 
$$\lim_{n\rightarrow \infty} \min_{T\in \Tset_{2n}}\avM(T)<
\lim_{n\rightarrow \infty}\avM(\PathT{n}).$$

For this purpose, let $T_n$ be the tree with $4n-2$ vertices obtained from $\Path{3n-2}$ by adding vertices $w_1,\ldots,w_n$ so that $w_i$ is linked to $v_{3i-2}$. Also, let $T_n^1$ be the graph obtained from $T_n$ by adding a vertex $u_1$ linked to $v_1$, and let $T_n^2$ be the graph obtained from $T_n^1$ by adding a vertex $u_2$ linked to $u_1$. For illustration, $T_{3}$, $T_{3}^1$ and $T_{3}^2$ are drawn in Figure \ref{fig:tree}.
	
	\begin{figure}[!ht]
		\centering
		\scalebox{0.95}{\includegraphics{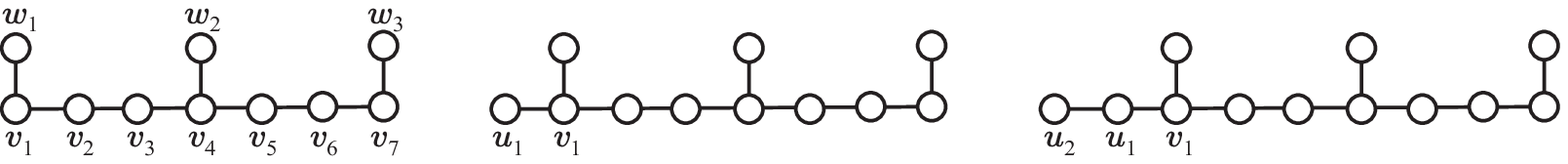}}
		\caption{$T_{3}$, $T_{3}^1$ and $T_{3}^2$.}
		\label{fig:tree}
	\end{figure}
\begin{thm}\label{thm:tree}
	$\lim\limits_{n\rightarrow \infty}\avM(T_n)=\frac{13}{18}.$
\end{thm}		
\begin{proof}
Let $M$ be a maximal matching of size $k$ in $T_n$. If $w_1v_1\in M$, then $M\setminus\{w_1v_1\}$ is a maximal matching in $T_{n-1}^2$. If $v_1v_2\in M$, then $M\setminus\{v_1v_2\}$ is a maximal matching in  $T_{n-1}^1$. We therefore have 
\begin{align}
\Sk{T_n}{k}=\Sk{T_{n-1}^1}{k-1}+\Sk{T_{n-1}^2}{k-1}.\label{eq:tree1}
\end{align}
Let now $M$ be a maximal matching of size $k$ in $T_{n}^1$. If $M$ contains $u_1v_1$ or $w_1v_1$, then the $k-1$ other edges of $M$ form a maximal matching in $T_{n-1}^2$. If $v_1v_2\in M$, then $M\setminus\{v_1v_2\}$ is a maximal matching in $T_{n-1}^1$. Hence, 
\begin{align}
\Sk{T_n^1}{k}=\Sk{T_{n-1}^1}{k-1}+2\Sk{T_{n-1}^2}{k-1}=\Sk{T_n}{k}+\Sk{T_{n-1}^2}{k-1}.\label{eq:tree2}
\end{align}
Finally, let $M$ be a maximal matching of size $k$ in $T_{n}^2$. If $u_1u_2\in M$, then $M\setminus\{u_1u_2\}$ is a maximal matching in $T_n$. If $u_1v_1\in M$, then $M\setminus\{u_1v_1\}$ is a maximal matching in $T_{n-1}^2$. The linear equation is therefore 
\begin{align}
\Sk{T_n^2}{k}=\Sk{T_{n}}{k-1}+\Sk{T_{n-1}^2}{k-1}.\label{eq:tree3}
\end{align}
Combining Equations (\ref{eq:tree1}) and (\ref{eq:tree2}) gives
\begin{align}
\Sk{T_n}{k}=&\Big(\Sk{T_{n-1}}{k-1}+\Sk{T_{n-2}^2}{k-2}\Big)+\Big(\Sk{T_{n-1}}{k-2}+\Sk{T_{n-2}^2}{k-2}\Big)\nonumber\\
=&\Sk{T_{n-1}}{k-1}+\Sk{T_{n-1}}{k-2}+2\Sk{T_{n-2}^2}{k-2}.
\end{align}
We can therefore rewrite Equation (\ref{eq:tree3}) as
\begin{align*}
&\Sk{T_{n{+}2}}{k{+}2}{-}\Sk{T_{n{+}1}}{k{+}1}{-}\Sk{T_{n{+}1}}{k}{=}2\Sk{T_{n}}{k{-}1}{+}\Big(\Sk{T_{n{+}1}}{k{+}1}{-}\Sk{T_{n}}{k}{-}\Sk{T_{n}}{k{-}1}\Big)\\
\Leftrightarrow& \Sk{T_{n}}{k}=2\Sk{T_{n-1}}{k-1}+\Sk{T_{n-1}}{k-2}-\Sk{T_{n-2}}{k-2}+\Sk{T_{n-2}}{k-3}.
\end{align*}
We have $I=2$, $\alpha_1=3$, $\alpha_2=0$, $\beta_1=4$ and  $\beta_2=1$, and $r=3$ is the root of maximum modulus of the equation $x^2-3x=0$. Since $\maxM(T_n)=2n+1$, we have $\lim_{n\rightarrow \infty}\frac{\maxM(T_n)}{n}=2$ and it follows that 
\begin{align*}\lim\limits_{n\rightarrow \infty}\avM(T_n)=\frac{4r+1}{2(3r)}=\frac{13}{18}\approx 0.7222.&\qedhere\end{align*}
\end{proof}
\noindent Note that $T_n$ has an even number of vertices, which implies$$\lim\limits_{n\rightarrow \infty} \min_{T\in \Tset_{2n}}\avM(T)\leq \frac{13}{18}<\frac{5+\sqrt{5}}{10}=
\lim\limits_{n\rightarrow \infty}\avM(\PathT{n}).$$
%
%
%
%We have enumerated all trees $T$ having up to 24 vertices, using \PHOEG~\cite{PHOEG}. We have thus determined those having the smallest value $\avM(T)$. For illustration we show in Figure \ref{fig:tree2} the optimal trees for $n$ ranging from $10$ to $24$. They are unique for each order, except for $n=13$ where two trees have the same minimal value.
%
%\begin{figure}[!h]
%	\centering
%	\scalebox{0.6}{\includegraphics{Tree2.eps}}
%	\caption{Trees of order $n\in\{10,\ldots,24\}$ with smallest value $\avM(T)$.}
%	\label{fig:tree2}
%\end{figure}
\subsection{Thorns of complete bipartite graphs}
We conclude our study with the thorn $\KT{c,n}$ of complete bipartite graphs $\K{c,n}$, where $c$ is a constant. Assume without loss of generality that $V_1,V_2$ is the partition of $\K{c,n}$ with $|V_1|=c$ and $|V_2|=n$. 
\begin{thm}\label{thm:bip1}
	$\lim\limits_{n\rightarrow \infty}\avM(\KT{c,n})=1$ for every integer constant $c\geq 0$.
\end{thm}
\begin{proof}
	Clearly, a maximal matching in $\KT{c,n}$ contains $i$ edges linking $V_1$ to $V_2$ ($0\leq i\leq c)$ and $(c-i)+(n-i)=n+c-2i$ edges linking a vertex $v$ to its neighbor $v'$ of degree 1. Hence, the number of maximal matchings of size $k=n+c-i$ is 
	$$\Sk{\KT{c,n}}{k}=\frac{c!n!}{(n+c-k)!(k-n)!(k-c)!}$$
	which implies
	$$\nM(\KT{c,n})=\sum_{k=n}^{n+c}\frac{c!n!}{(n+c-k)!(k-n)!(k-c)!}\quad  \mbox{and}  \quad \totM(\KT{c,n})=\sum_{k=n}^{n+c}\frac{kc!n!}{(n+c-k)!(k-n)!(k-c)!}.$$
	Hence, $\lim_{n\rightarrow \infty}\nM(\KT{c,n})\sim n^c$ and $\lim_{n\rightarrow \infty}\totM(\KT{c,n})\sim n^{c+1}$ which implies
\begin{align*}\lim\limits_{n\rightarrow \infty}\avM(\KT{c,n})=\lim\limits_{n\rightarrow \infty}\frac{\totM(\KT{c,n})}{\maxM(\KT{c,n})\nM(\KT{c,n})}=\lim\limits_{n\rightarrow \infty}\frac{n^{c+1}}{(n+c)n^c}=1.&\qedhere\end{align*}
\end{proof}

\section{Concluding remarks} \label{sec_conclusion}

It is well known that the size of a maximal matching in a graph $G$ is at least half of the size $\maxM(G)$ of a maximum matching in $G$. There are graphs in which almost all maximal matchings are maximum while for others, almost all maximal matchings have half of the size of a maximum matching. This prompted us to investigate the ratio $\avM(G)$ of the average size of a maximal matching to the size of a maximum matching in $G$. Clearly, $\avM(G)\approx 1$ if many maximal matchings have a size close to $\maxM(G)$,  while $\avM(G)\approx \frac{1}{2}$ if many maximal matchings have a small size. 

We have determined $\lim_{n\rightarrow \infty}\avM(G_n)$ for many families $\{G_n\}_{n\geq 0}$ of graphs, showing that intermediate values between $\frac{1}{2}$ and $1$ are reached. We show eighteen of them in Figure \ref{fig:list}. These values were obtained using a general technique that can be applied when the number of maximal matchings in $G_n$ linearly depends on the number of maximal matchings in graphs $G_{n'}$ of the same family, with $n'<n$. This technique allowed us to recalculate known values of $\lim_{n\rightarrow \infty}\avM(G_n)$ which were obtained by other authors using different concepts such as (bivariate) generating functions. 

A comparison of the curve in Figure \ref{fig:list} with the bottom curve of Figure \ref{fig:Comparison} demonstrates that the asymptotic values $\lim_{n\rightarrow \infty}\avM(G_n)$ for the classes of graphs studied in Section \ref{sec_families} are representative of the values of $\avM(G)$ for any graphs $G$.

\begin{figure}[!ht]
	\centering
	\scalebox{0.85}{\includegraphics{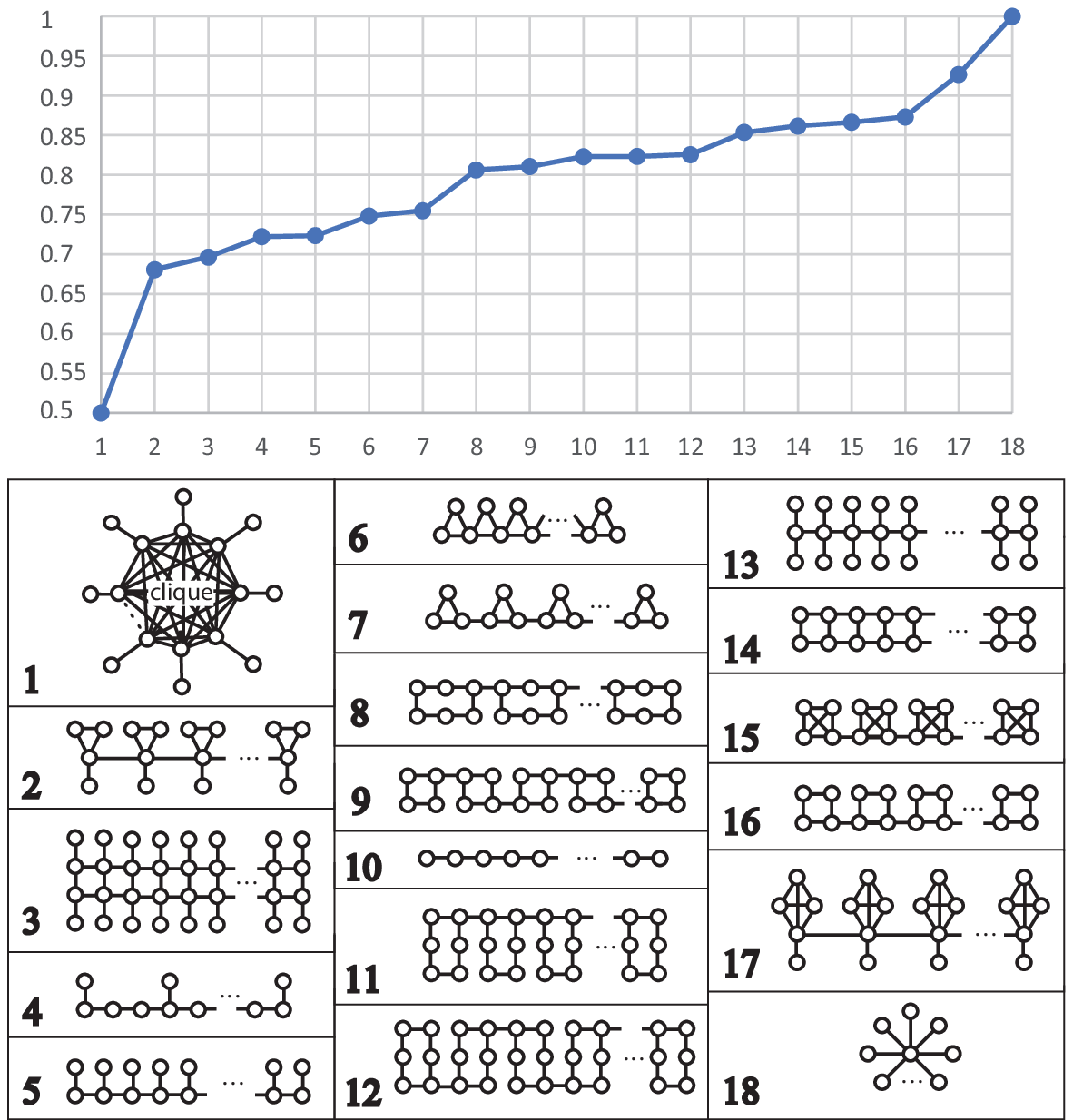}}
	\caption{Various families $\{G_n\}_{n\geq 0}$ of graphs  and their asymptotic value $\lim\limits_{n\rightarrow \infty}\avM(G_n)$.}
	\label{fig:list}
\end{figure}

We conclude the paper with an open problem on the class $\Tset_{2n}$ of trees of even order. Dyer and Frieze \cite{DF} conjecture that 
$\lim_{n\rightarrow \infty} \min_{T\in \Tset_{2n}}\dyer(T)=
\lim_{n\rightarrow \infty}\dyer(\PathT{n}).$ In simpler words, they think that the worst case for  $\dyer$ on $\Tset_{2n}$ are the graphs obtained from $\Path{n}$ by adding a new vertex $v'$ for every $v$ of $\Path{n}$ and linking $v$ to $v'$. We have shown in Section \ref{sec:trees} that $\lim_{n\rightarrow \infty} \min_{T\in \Tset_{2n}}\avM(T)<
\lim_{n\rightarrow \infty}\avM(\PathT{n}).$ Hence, the worst case for $\avM$ on $\Tset_{2n}$ seems to be different from that of $\dyer$ on $\Tset_{2n}$ and it would be interesting to characterize it.

\section*{Acknowledgments}

The authors would like to thank Richard Labib for his help, an anonymous referee for the elegant proof of Theorem \ref{thm1}, and another anonymous referee for useful chemical references.

\bibliographystyle{acm}
\bibliography{avMatch}

\end{document}